\documentclass{article}
\usepackage{amssymb}
\usepackage{soul} 
\usepackage{amsmath}
\usepackage{amsthm} 
\usepackage{hyperref}
\usepackage{xcolor}
\newtheorem{theorem}{Theorem}

\newtheorem{lemma}{Lemma}
\newtheorem{proposition}{Proposition}

\newtheorem{remark}{Remark}

\newtheorem{definition}{Definition}

\newcommand{\R}{\mathbb{R}}

\renewcommand\P{\operatorname{\mathbf{P}}}
\newcommand\E{\operatorname{\mathbf{E}}}
\newcommand{\X}{\mathbf{X}}
\newcommand{\x}{\mathbf{x}}

\begin{document}

\title{Two sided long-time optimization singular control problems for L\'evy processes and Dynkin's games}

 \author{Ernesto Mordecki\thanks{Centro de Matem\'atica. Facultad de Ciencias, Universidad de la Rep\'ublica, Montevideo, Uruguay. . email: mordecki@cmat.edu.uy} \and Facundo Oli\'u\thanks{Ingenier\'ia Forestal. Centro Universitario de Tacuaremb\'o, Universidad de la Rep\'ublica, Uruguay,  email: facundo.oliu@cut.edu.uy}. \thanks{ Department of Mathematics, Christian-Albrechts-Universität zu Kiel, Germany.}}

\maketitle

\abstract{A relationship between two sided discounted  singular control problems and Dynkin games is established for real valued L\'evy processes. To be more specific, we state a  verification Theorem for the discounted problem in the form of a Hamilton-Jacobi-Bellman equation and  prove that the primitive of the value function of a Dynkin game satisfies the necessary hypotheses. This link allows us to study regularity properties and define optimal controls for the discounted singular problem. 
Moreover, we prove that the abelian limit holds and as a consequence, the solution of a two-sided ergodic singular control problem is obtained as the limit of the corresponding solution of the discounted one.
With these results, we conclude that the optimal controls within the class of c\`adl\`ag strategies can be in fact found 
in the class of reflecting barriers controls for both problems, and solved through some deterministic equations.
To illustrate the results, three examples are given: 
compound Poisson processes with two-sided exponential jumps with and without Gaussian component, and stable processes.  }


\vskip 5pt


\section{Introduction}

Singular stochastic control problems offer a versatile framework for addressing various economic and financial decision-making problems. These problems involve optimizing decisions in the presence of uncertainty, with applications ranging from harvesting policies to cash flow management.
%
For example, one instance of a singular stochastic control problem is maximizing the expected cumulative present value of harvesting yield when the harvesting effort is unbounded. 
This specific problem has been extensively studied and analyzed in the literature 
(see for example \cite{AS,HNUW, LESE,LO}).
Another application, focusing on cash flow management, studies optimal dividend distribution, recapitalization, or a combination of both while considering risk neutrality. These investigations also fall under the domain of singular control problems (see \cite{AT, HT, JS, BK,P,SLG}).

In the present paper, we consider two sided control problems for an underlying L\'evy processes, 
having both ergodic and discounted accumulated rewards.
Our aim is to find the value functions of the problem and the best strategies within a class of admissible controls.
Regarding different types of problems that can be considered, 
it should be taken into account (i) the underlying process, in our discussion is a (eventually multidimensional) diffusion or a L\'evy process;
(ii) the type of singular control applied: one sided (we have one increasing process at our disposal to add/subtract to the underlying) or two-sided (we have a pair of increasing processes to add and subtract to the underlying), and (iii) the accumulated cost/reward, that can be ergodic or discounted.

The main contributions of the present paper consist in showing that the solution of the discounted problem above can be obtained through the solution of an associated Dynkin game: 
the continuation region of the Dynkin game is an interval whose extremes define a two sided reflecting optimal strategy for the discounted problem. 
With this result in view, we prove that the abelian limit holds. This
means that, when the discount rate $\epsilon$ decreases to zero,
the normalized expected reward associated with the $\epsilon$-discounted problem converges to a constant for each starting point, 
the latter constant being actually the  value of the ergodic problem. With this technique, we prove that there is an optimal reflecting strategy for the ergodic problem.
Moreover, we prove that the barriers of the optimal controls have a convergent subsequence whose limit defines optimal barriers for the ergodic problem. 
Our approach happens to be a \emph{direct} approach, in the sense that no guessing is necessary,
the Dynkin game can be solved by finding the optimal thresholds.


General results (such as the existence of solutions) for the types of problems considered in this paper are obtained in the general framework of Markov processes in \cite{LukaszStettner}.
The relation of control problems and Dynkin games was obtained either for other class of processes or in a more general context under more stringent conditions in \cite{KW,GT}. 
In addition, we obtain smoothness properties of the value function, which allows us to relate the discounted problem to the ergodic one.
The proof of the abelian limit for diffusions was obtained  in \cite{Alvarez,AZ}. 
For more references on this classical relationship, see \cite{YuriKifer,TG,TF2,DF,SFR}. 
For one-sided controlled problems,  a link between optimal control and optimal stopping problems has been established. For instance, in \cite{KARATZAS} the underlying process is a Brownian motion, 
or more recently, in \cite{NY,SEXTON} L\'evy processes are considered.
In relation to the determination of optimal boundaries of the ergodic problem
departing from the discounted ones, a continuity criterion was used in \cite{HDF}, 
(we use our verification theorems in such a way that they relate both problems).
%
In the article \cite{BAURDOUX} two sided controlled problems for L\'evy processes with one-sided jumps are considered.
For the case where an It\^{o} diffusion is controlled by a process of bounded variation, 
smoothness properties of the value function have been studied in  \cite{GT0}, 
and verification theorems in the form of Hamilton Jacobi Bellman (HJB) equations that define a free boundary problem have been proposed and solved explicitly (see for example, \cite{FV,KKXYZ,CMO}). 
In these cases, the optimal control problem reduces to the question of finding two barriers such that the process is reflected when it reaches them.
Finally, regarding reflected L\'evy processes, the tools needed to write down our results (such as the existence and form of the double reflected L\'evy process) can be found in \cite{AAGP,KRLS}.

%

Let us mention some specific issues regarding the consideration of L\'evy processes.
The main difficulty is the overshoot. When the process is stopped or reflected,
we can not keep it within an interval, as the jumps are not predictable.
In terms of equations, this is related with the fact that the 
infinitesimal operator $\mathcal{L}$ is not local (i.e. has an integral term), so that in order to find solutions of the equations $\mathcal{L}u=\epsilon u$, values of $u$ in the whole real line should be considered. This makes it unfeasible to solve the above equation in an HJB problem.
It does not seem clear how to prove directly that the HJB has a solution with the needed regularity,
as is done in the case of diffusions. This difficulty is partially overcome with the consideration of spectrally one sided L\'evy processes (which allows to the use of the scale function)
or the consideration of tractable jump measures, related to exponential distributions
(see the discussion and the references in \cite{AMS}).

The remainder of the paper is organized as follows. 
In Section \ref{S:framework} we give definitions and present the main results.
In Section \ref{S2:DynkinAdjoint} the relation with the associated Dynkin game is established
for the discounted problem, and Section \ref{S3:Candidate} 
contains the construction of the optimal controls and the proof of the main
results.
Finally, Section \ref{S4:Examples} presents three examples:  the first two when the driving L\'evy processes are Compound Poisson process with two sided exponential jumps with and without Gaussian component, the third involving a strictly stable process with finite mean.
The results obtained in these examples are possible because the two-barrier problem for these processes can be solved explicitly.

\section{Definitions and main results}\label{S:framework}

\subsection{L\'evy processes, strategies and cost functions}

Let $X=\{X_t\}_{t\geq 0}$ be a L\'evy process with finite mean defined on a stochastic basis 
${\cal B}=(\Omega, {\cal F}, {\bf F}=({\cal F}_t)_{t\geq 0}, \P_x)$ departing from $X_0=x$.
Assume that the filtration is right-continuous and complete (see \cite[Def. 1.3]{JJAS}).
Denote by $\E_x$ the expected value associated to the probability measure $\P_x$, let $\E= \E_0$ and $\P=\P_0$.   
The L\'evy-Khintchine formula characterizes the law of the process,
stating
$$
\phi (z)= \log \left( \E  e^{z X_1} \right), \qquad z=i\theta \in i\mathbb{R},
$$
with
\begin{equation*}
\phi(z)={\nu^2\over 2}z^2+z \mu+\int_{\R}\left(e^{z y}-1-z y\right)\Pi(dy),
\end{equation*}
where $\mu=\E X_1\in\R$, $\nu\geq 0$ and $\Pi(dy)$ is a non-negative measure (the \emph{jump measure}) 
that satisfies in our case, due to the assumption that our process has finite mean, $\int_{\R}(y^2\wedge |y|)\Pi(dy)<\infty$. 
This L\'evy process, being a special semimartingale (see \cite[II.2.29]{JJAS}), 
can be expressed as a sum of three independent processes 
\begin{equation}\label{D:LEVYPROCESSEQUATION}
X_t= X_0 +\mu t +\nu W_t + \int_{[0,t]\times\R} y\, \tilde N(ds,dy),
\end{equation}
where $\tilde N(ds,dy)=N(ds,dy)-ds\Pi(dy)$ is a compensated Poisson random measure, 
$N(ds,dy)$ being the jump measure constructed from $X$ (see \cite[II.1.16]{JJAS}),
and $\lbrace W_t \rbrace_{t \geq 0}$ is an independent Brownian Motion. 

In the case where $X$ has bounded variation, we introduce the drift (see \cite[II, Lemma 2.13]{KRI}) as 
\begin{equation}\label{eq:driftboundedvariation0}
\tilde \mu = \mu - \int_{\mathbb{R}} y \Pi(dy),
\end{equation} 
 and introduce
\begin{align}
 S_t^+ &= \max (0,\tilde \mu) t + \int_0^t \int_{(0, \infty)} x \, N(ds, dx), \label{eq:esemas} \\
S_t^-& = \max(0,- \tilde \mu) t + \int_0^t \int_{(-\infty, 0)} |x| \, N(ds, dx),\label{eq:esemenos}
\end{align}
that are a couple of independent subordinators starting from zero, such that
$
X_t=x +S_t^+ -S_t^-,
$
 for all $t \geq 0$.

The infinitesimal generator of the process $X$ is 
\[
\mathcal{L}f (x)=\lim_{t \to 0} \frac{\displaystyle \E_x f(X_t)-f(x)  }{\displaystyle t} , \]
defined for every real function such that the limit exists for every $x \in \mathbb{R}$. 
We say that such a function is in the domain of the infinitesimal generator. For general references on L\'evy processes, see \cite{KRI, KIS}.
\begin{definition}\label{D:Admissiblestrategies} 
An \emph{admissible control} is a pair of non-negative $\mathbf{F}$-adapted processes $(U,D)$ such that:
\vskip1mm\par\noindent
{\rm(i)} Each process $U,D\colon\Omega \times \mathbb{R}_+ \rightarrow \mathbb{R}_+$ is right continuous and nondecreasing almost surely.
\vskip1mm\par\noindent
{\rm(ii)} For each $t\geq 0$, the random variables $U_t$ and $D_t$ have finite expectation.
\vskip1mm\par\noindent
We denote by $\mathcal{A}$ the set of admissible controls.
\end{definition}
A controlled L\'evy process by the pair $(U,D) \in \mathcal{A}$ is be defined as
\begin{equation}\label{D:controlledequation}
X^{U,D}_t= X_t +U_t-D_t, \qquad X_0=x,\ U_{0}=u_0,  \ D_{0}=d_0.
\end{equation}
For $a<b$, let $\lbrace X^{a,b}_t\colon \ t \geq 0 \rbrace$ be a process defined on $(
\Omega,\mathcal{F}, {\bf F},\P_x)$ that follows \eqref{D:controlledequation} where 
$ U^{a,b}_t,-D^{a,b}_t$ are the respective reflections at $a$ and $b$,  called \emph{reflecting controls}.
At time zero, the reflecting controls are defined as $U^{a,b}_0=(a-x)^+$ and $D^{a,b}_0=(x-b)^+$ (where $x^+= \max(x,0)  $) and from now on are denoted $u_0^{a,b}$ and $d_0^{a,b}$ respectively. 
For the existence, uniqueness and construction of the reflecting controls for processes with c\`adl\`ag trajectories,
we refer to  \cite[Def. 1.2]{kroketa}.
The controlled process and the controls are uniquely defined by the conditions 
$a\leq X^{a,b}_t\leq b$ for all $t>0$ a.s.,  and
\begin{equation*}
\int_{[0,\infty)} (X^{a,b}_t -a)dU^{a,b}_t=0, \quad   \int_{[0,\infty)} (b- X^{a,b}_t )dD^{a,b}_t=0 . 
\end{equation*}
The particular case of L\'evy processes was considered in \cite{AAGP}. 

We note that for every exponential random variable $e(\epsilon)$ with parameter $\epsilon>0$ independent of the process $X$ and every $t>0$ the random variables $U^{a
,b}_{e(\epsilon)} \ U^{a,b}_t, \ D^{a,b}_{e(\epsilon)}, \  D^{a,b}_t, $ have finite mean as a consequence of \cite[Thm. 6.3]{AAGP}, (more specifically, see Proposition \ref{P:notdegenerate}). When the process has bounded variation, we also define 
$$(U^{0,0}_0,D^{0,0}_0)= (-\min \lbrace x,0 \rbrace , \max \lbrace x,0 \rbrace), \quad (U^{0,0}_t,D^{0,0}_t)=(S^-_t,S^+_t ), \text{ for } t>0,  $$
as a reflecting strategy, with the processes defined in \eqref{eq:esemas} and \eqref{eq:esemenos}.
Observe that formula \eqref{D:controlledequation}, in this case, holds with $X_t^{U^{0,0},D^{0,0}}=0$, for all $t \geq 0$. 
 
From now on $q_u, q_d $ are positive constants, 
and we refer to them as lower barrier cost and upper barrier cost, respectively. Also, denote $q=q_u+q_d$.
\begin{definition}\label{D:costfunction} 
A cost function is a convex non negative function  
$c\colon \mathbb{R} \rightarrow \mathbb{R}_+$  
such that
\vskip1mm\par\noindent
{\rm(i)}  reaches its minimum at zero, 
\vskip1mm\par\noindent
{\rm(ii)} there are three constants $\alpha\geq 0, M \geq 0$ and $N>0$ that satisfy 
\begin{equation*}
 c(x) +M \geq   N\vert x \vert^{1+\alpha},  \qquad \text{for all $x \in \mathbb{R}$},
\end{equation*}
\vskip1mm\par\noindent
{\rm(iii)} for every $\delta>0$ there is a convex function 
$c_\delta \in C^2 (\mathbb{R})$ 
with minimum at zero such that 
$\| c- c_{\delta} \|_{\infty} < \delta$ and for every $\epsilon>0, \ x \in \mathbb{R}$
\begin{equation}\label{eq:cprima}   
\E_x \int_0^{\infty} \vert c_{\delta}'(X_s)  \vert e^{-\epsilon s} ds  < \infty . 
\end{equation}
\end{definition}

\begin{remark}\label{Remark:1}
Condition \eqref{eq:cprima} is not straightforward to verify. 
First, in the case where there is a constant $n \in \mathbb{N}$ such that $\vert c'_\delta(x)\vert \leq A_{\delta}+ \vert x \vert^n$ and the process $X$ has finite first $n$ moments then the condition holds, because $\E_x \vert c_{\delta}' (X_s) \vert$ is bounded by the absolute value of a polynomial in $s$ (due to Burkholder-Rosenthal inequality). 
Nevertheless, condition \eqref{eq:cprima} holds if    there are constants $A,B,\theta>0$ such that $\vert c'_{\delta}(x)\vert \leq A+B\cosh (\theta x)$ and $\max \left(\phi (\theta), \phi(-\theta)\right)< \epsilon$ for every $\delta>0$.
\end{remark}
\subsection{The ergodic and discounted control problems}\label{S:theproblem}
\begin{definition}\label{D:ergodicvaluefunction}
Given $x\in\R$ and a control $(U,D)\in\mathcal{A}$, we define the ergodic cost function 
\[
J(x,U,D) = \limsup_{T \to \infty} \frac{1}{T} \E_x\left( \int_{(0,T]}c(X^{U,D}_s)ds +q_u U_T +q_d D_T  \right), 
\]
and the ergodic value function 
\[
G(x) = \inf_{(U,D) \in \mathcal{A}} J(x,U,D).
\]
\end{definition}

\begin{definition}\label{D:discountedgame}
Given $x\in\R$, a control $(U,D)\in\mathcal{A}$ and a fixed $\epsilon>0$, 
we define the $\epsilon$-discounted cost function 
\[
J_{\epsilon}(x,U,D)  =  \E_x\left(\int_{(0,\infty)}e^{-\epsilon s}\left(c(X_s^{U,D})ds + q_udU_s +q_d dD_s\right)+q_uu_0+q_dd_0\right), 
\]
and the $\epsilon$-discounted value function
\[
G_{\epsilon}(x)= \inf_{(U,D) \in \mathcal{A}} J_{\epsilon}(x,U,D).
\]
\end{definition}  

\subsection{Formulation of main results}
 The most important results of the article are the link between the discounted problem and a Dynkin game, the optimality of reflecting strategies for the discounted problem  (see Theorem \ref{T:DISCOUNTEDPROBLEMSOLUTION}) and the abelian limits which give an optimal reflecting strategy for the ergodic problem (see Theorem \ref{T:ErgodicProblemsolution}). 
 As will be seen in the following theorems, the case $\alpha=0$ only has relevance in the ergodic case.

\begin{theorem}\label{T:DISCOUNTEDPROBLEMSOLUTION}
Under the same notations as in Def. \ref{D:discountedgame}, 
if $\epsilon$ is small enough in case $\alpha=0$, or if $\alpha>0$,
there is a pair $a^{\ast}_{\epsilon}\leq 0\leq b^{\ast}_{\epsilon}$ such that $G_{\epsilon}(x)=J_{\epsilon}(x,U^{a^{\ast}_{\epsilon},b^{\ast}_{\epsilon}},D^{a^{\ast}_{\epsilon},b^{\ast}_{\epsilon}})$. 
Moreover, if $c \in C^2(\mathbb{R})$, then $a^{\ast}_{\epsilon}<0<b^{\ast}_{\epsilon}$, $G_{\epsilon} \in C^1(\mathbb{R})$ ($G_{\epsilon} \in C^{2}(\mathbb{R})$ if the process has unbounded variation) and the function
\begin{multline*}
    V_{\epsilon}(x)   = \sup_{a \leq 0} \inf_{b \geq 0}\E_x \left(\int_0^{\tau(a) \wedge \sigma(b)} c'(X_s)e^{-\epsilon s}ds  \right. \\+q_d e^{-\epsilon \tau(a)} \mathbf{1}_{ \lbrace\tau(a) \leq \sigma(b)\rbrace } -q_ue^{-\epsilon \sigma(b)} \mathbf{1}_{\lbrace\sigma(b) <\tau(a)\rbrace}  \Bigg) ,
    \end{multline*}
with
\begin{equation}\label{D:tauandsigma}
    \tau(a)=\inf\{t\geq 0\colon X_t\leq a\},
\quad
\sigma(b)=\inf\{t\geq 0\colon X_t\geq b\}, 
\end{equation}
satisfies $V_{\epsilon}=G_{\epsilon}'$.
Furthermore $ a^{\ast}_{\epsilon},b^{\ast}_{\epsilon}$   satisfy:
\begin{equation*}
a^{\ast}_{\epsilon}=\sup \{x<0, V_{\epsilon}(x)=-q_u \},\ 
b^{\ast}_{\epsilon}=\inf \{x>0, V_{\epsilon}(x)=q_d \}    
\end{equation*}
and realize the saddle point in $V_\epsilon (x)$. In other terms 
\begin{multline*}
V_\epsilon (x)=\E_x \Bigg(\int_0^{\tau(a^\ast_\epsilon) \wedge \sigma(b^\ast_\epsilon)} c'(X_s)e^{-\epsilon s}ds  \\
	+q_d e^{-\epsilon \tau(a^\ast_\epsilon)} \mathbf{1}_{ \lbrace\tau(a^\ast_\epsilon) \leq \sigma(b^\ast_\epsilon)\rbrace } -q_ue^{-\epsilon \sigma(b^\ast_\epsilon)} \mathbf{1}_{\lbrace\sigma(b^\ast_\epsilon) <\tau(a^\ast_\epsilon)\rbrace}  \Bigg). 
\end{multline*}
\end{theorem}

\begin{theorem}\label{T:ErgodicProblemsolution} 
The ergodic value function $G$ of Definition \ref{D:ergodicvaluefunction} satisfies:
\vskip1mm\par\noindent
{\rm(i)} 
$\lim_{\epsilon  \searrow 0}\epsilon G_{\epsilon}  = G$ uniformly in compacts.
\vskip1mm\par\noindent
{\rm(ii)}
The ergodic value function is constant for all $x\in\R$.
\vskip1mm\par\noindent
{\rm(iii)}
There is a sequence $\lbrace (a^\ast_{\epsilon_n},b^\ast_{\epsilon_n}, \epsilon_n) \rbrace_{n \in \mathbb{N}}$ converging to $(a^{\ast},b^{\ast},0) $ such that $$ \lim_{n \to \infty}\epsilon_n G_{\epsilon_n}(x) =\lim_{n \to \infty}\epsilon_n J_{\epsilon_n}(x, U^{a^{\ast}_{\epsilon_n},b^{\ast}_{\epsilon_n}},D^{a^{\ast}_{\epsilon_n},b^{\ast}_{\epsilon_n}})=J(x,U^{a^{\ast},b^{\ast}},D^{a^{\ast},b^{\ast}}) =G(x)$$ for all $x\in \mathbb{R}$. 
\end{theorem}

\begin{remark}\label{Remark:multidimensional}
As seen above, our strategy to solve the ergodic case is to prove an  Abelian limit. 
Nevertheless, during the preparation of this manuscript, it has come to our attention the work \cite{ALESSANDROFEDERICOGIORGIO}, where a diffusion driven by a multi-dimensional Wiener process, 
additionally aﬀected by another multi-dimensional non-controlled process
 is considered as the underlying process.
The authors establish a useful connection between the ergodic problem and a Dynkin game. 
The techniques employed, in principle, do not seem to be easily adaptable to our setting (see Hypothesis 2.2 of the mentioned paper). 
\end{remark}

\subsection{The Hamilton-Jacobi-Bellman formulation}\label{S1:Verification}
In order to solve the optimal control problems for the cost functions defined above, 
the usual approach,
when the process is an It\^o diffusion,  
is to formulate a verification theorem with a Hamilton-Jacobi-Bellman (HJB) equation, as done in \cite{SLG,H}, 
or more recently, 
for two sided problems, in \cite{KKXYZ,CMO}). 
We use a similar approach for L\'evy process. First, we prove some basic properties of the value functions.

\begin{proposition}\label{P:notdegenerate}
For every $\epsilon>0$, the functions $G_{\epsilon}$ and $G$  are finite. Moreover, $G_{\epsilon}$ is upper semicontinuous.    
\end{proposition}
\begin{proof}
For the ergodic value, finiteness is established by taking a reflecting control and using  \cite[Cor. 6.6]{AAGP}. 
For the $\epsilon$-discounted value, from eq. (70) in \cite[Thm. 6.3]{AAGP},  
we have for $(U^{0,b},D^{0,b}), \ b >0$: 
\begin{multline*}
    b D^{0,b}_t  \leq b^2+x^2  +2\int_{(0,t)} X^{0,b}_{s^-}dX_s+\frac{\nu^2}{2}t 
    + \sum_{s \leq t} \left(\mathbf{1}_{\lbrace\Delta X_s \geq b\rbrace} (2 \Delta X_s b+b) 
    \right. \\ \left. + \mathbf{1}_{\lbrace\Delta X_s \leq -b\rbrace}(b^2 -2b \Delta X_s)+ \mathbf{1}_{\lbrace\vert \Delta X_s \vert <b\rbrace } \Delta X_s^2 \right).
\end{multline*}
Therefore, observe that the process 
$$
t \rightarrow \int_{(0,t)} X^{0,b}_{s^-}d(X_s-s\E X_1),
$$ 
is a martingale. This is deduced by decomposing $X$ as the sum of two independent L\'evy processes, one with second moments, using \cite[Thm. IV.2.11]{ProtterPE} and the other a Compound Poisson process. Thus, by taking expectations:
\begin{multline}\label{eq:reflectingbound}
 b \E_{x}D^{0,b}_t  \leq b^2+x^2  +2b t\E \vert X_1 \vert +\frac{\nu^2}{2}t \\ + t\int_{\mathbb{R}}\left(\mathbf{1}_{\lbrace y \geq b \rbrace} (2 y b+b) 
    + \mathbf{1}_{\lbrace y \leq -b \rbrace}(b^2 -2b y)+ \mathbf{1}_{\lbrace \vert y \vert <b \rbrace } y^2 \right)\Pi(dy).
\end{multline}
Therefore, there are a couple of constants $K_1, K_2$ such that $ \E_x D_t^{0,b} \leq K_1 +K_2 t$. Thus
\begin{multline*}
q_d\E_x \left(\int_0^{\infty}e^{-\epsilon s}dD_s+d_0 \right) \\ =q_d\E_x \left( \lim_{t \to \infty} e^{-\epsilon t} D_t^{0,b} +\int_{0}^{\infty}D_{s^-} \epsilon e^{-\epsilon s}ds \right)\leq q_d(K_1 +K_2/\epsilon)< \infty. 
\end{multline*}
With a similar argument for $U^{0,b}$ and the fact that $c$ is bounded when restricted to an interval, we deduce the finiteness of $J_{\epsilon}(x,U^{0,b},D^{0,b})$.

For the upper continuity, first notice that
\begin{equation*}
    G_{\epsilon}(x)=  \inf_{(U,D)\in \mathcal{A}} \E\left( \int_{(0,\infty)}  e^{-\epsilon s}\left(c(x+X_s^{U,D})ds +q_udU_s +q_d dD_s\right)  +q_uu_0+q_dd_0\right). 
\end{equation*}
Therefore, for this proposition, we only work with the probability measure $\P$.
Now fix $x \in \mathbb{R}$ and $r>0$. 
Take  $(U,D)\in\mathcal{A}$ such that  
$$ 
J_{\epsilon}(x,U,D)-r\leq G_{\epsilon}(x).
$$ 
Take $y \in \mathbb{R}$ and let $(U^y,D^y) \in \mathcal{A}$ be defined as
\begin{equation*}
    D^y_t=D_t+(y-x)^+ ,  \ U^y_t=U_t+(x-y)^+,  \ \text{for all } t \geq 0.
\end{equation*}
Notice that $(U^y,D^y)$ are right continuous adapted processes and they only differ by a constant from $(U,D)$. 
It is clear $J_{\epsilon}(x,U,D)- J_\epsilon(y,U^y,D^y)=-(q_u (x-y)^+ +q_d (y-x)^+)$. 
Therefore, 
$$
\limsup_{y \to x} G_{\epsilon}(y) \leq \limsup_{y \to x} J_{\epsilon}(y,U^y,D^y) = J_{\epsilon}(x,U,D)\leq G_{\epsilon}(x)+r. 
$$
The proof is concluded because $r$ is arbitrary.
\end{proof}
To avoid redundancy in the hypotheses of the theorems of this section, we need the following standard result, which can be proved using different variants of It\^o's formula (see \cite[II.7]{ProtterPE}).
\begin{definition}
    We say that a function is linear outside an interval $[a,b]$ when it is differentiable
    on $(a,b)^c$ and $u'(x)=u'(a)$ for all $x\leq a$ and $u'(x)=u'(b)$ for all $x \geq b$.
\end{definition}

\begin{proposition}\label{P:technicalinfinitesimal}
 Consider an interval $[a,b]$. \\ 
{\rm(i)} Assume that $X$ has unbounded variation.
A function $u \in C^2(\mathbb{R})$ that is linear outside $[a,b]$
is in the domain of the infinitesimal generator, 
$\mathcal{L}u$ is continuous, and
\begin{equation*}
\mathcal{L}u(x) =\mu u'(x)+\int_{\mathbb{R}}\left(u(x+y)-u(x)-yu'(x)\right) \ \Pi(dy) +\frac{\nu^2}{2}u''(x).
\end{equation*}
{\rm(ii)}
On the other hand, if the process $X$  has bounded variation, 
a function $u \in C^1(\mathbb{R})$ 
that is linear outside $[a,b]$
is in the domain of the infinitesimal generator, $\mathcal{L}u$ is continuous, and
\begin{equation*}
\mathcal{L}u(x)= \mu u'(x)+\int_{\mathbb{R}}\left(u(x+y)-u(x)-yu'(x)\right)\Pi(dy).
\end{equation*} 
\end{proposition} 

\begin{theorem}\label{T:Verification}
Consider  
$\epsilon > 0,\delta > 0$, a cost function $c$ (with associated function $c_{\delta}$), and a convex function 
$u $ in the domain of the infinitesimal generator such that
\begin{equation}\label{eq:conditions}
\mathcal{L}u (x)-\epsilon u(x)+ c_{\delta}(x)  \geq 0, \ -q_u \leq u'(x) \leq q_d,\quad\text{for all $x \in\mathbb{R}$}. 
\end{equation}
Furthermore, assume that
there exists an interval such that $u$ is linear in its complement and \vskip1mm\par\noindent{\rm(i)}
if the process is of unbounded variation, $u \in  C^{2}(\mathbb{R})$.
\vskip1mm\par\noindent{\rm(ii)}
if the process is of bounded variation, $u \in  C^{1}(\mathbb{R})$.
\vskip2mm\par\noindent
Then, under these conditions:
\begin{equation*}
\liminf_{T \to \infty} \frac{\epsilon}{T}\E_{x}   \int_{(0,T]}
 u(X^{U,D}_s)ds  \leq 
  \liminf_{T \to \infty}   \frac{1}{T} \E_{x}\left(\int_{(0,T]}
   c_{\delta}(X^{U,D}_s)ds +q_u U_T +q_d D_T \right),
\end{equation*}
for all controlled processes $X^{U,D}$ that satisfy the condition
\begin{equation}\label{eq:allcontroled}
\liminf_{T \to \infty} \frac{1}{T} \E_{x}u(X^{U,D}_T)=0. 
\end{equation}
\end{theorem}
\begin{proof}
Consider first case (i). Define the martingales $\{M^{(i)}_t\}_{t\geq 0}\ (i=1,2)$ by
\begin{align*}
    M^{(1)}_t&=\nu W_t+\int_{(0,t]\times \R}y\,\tilde N(ds,dy),\\
    M^{(2)}_t&=\int_{(0,t]\times \R}\left(
    u(X^{U,D}_{s^-}+y)-u(X^{U,D}_{s^-})-yu'(X^{U,D}_{s^-})
    \right)\tilde N(ds,dy).
\end{align*}
%
Observe that $\langle(X^{U,D})^{c}, (X^{U,D})^{c}\rangle_t=\langle X^{c}, X^{c}\rangle_t=\nu^2 t$  (see \cite[Lem. 6.2]{AAGP}) and apply It\^o's formula (the following sums can be separated due to $u''$ being bounded):
\begin{align}
    u(X^{U,D}_{T})&-u(x+u_0-d_0) \notag \\ &=\int_{(0,T]}u'(X^{U,D}_{t^-})dX^{U,D}_t+\frac{\nu^2}{2}\int_{(0,T]}u''(X^{U,D}_{t^-})\,dt+M^{(2)}_T\notag\\
    &\quad+\int_{(0,T]\times \R}\left(
    u(X^{U,D}_{t^-}+y)-u(X^{U,D}_{t^-})-yu'(X^{U,D}_{t^-})
    \right)ds\Pi(dy)\notag\\
    &\quad +\sum_{0<s \leq t}\left( u(X^{U,D}_{s})-u(X^{U,D}_{s^-}+\bigtriangleup X_s)-u'(X_{s^-}^{U,D})(\bigtriangleup U_s -\bigtriangleup D_s) \right) \nonumber \\
    &=\int_{(0,T]}\mathcal{L}u(X^{U,D}_{t^-})\,dt+\int_{(0,T]}u'(X^{U,D}_{t^-})\,dM^{(1)}_t+M^{(2)}_T\notag\\
    &\quad+\int_{(0,T]}u'(X^{U,D}_{t^-})(dU^c_t-dD^c_t)\notag\\
    & \quad +\sum_{0<s \leq t} \left( u(X_s^{U,D})-u(X_{s^-}^{U,D}+\bigtriangleup X_s) \right)  \notag \\
    &\geq \int_{0}^T\left(\epsilon u-c_\delta\right)(X^{U,D}_{t^-})\,dt-q_u(U_{T}-u_0) -q_d (D_{T}-d_0)+M^{(3)}_{T}.\label{useful}
\end{align}
where we used both conditions in \eqref{eq:conditions}, $M^{(3)}$ is a martingale and the equality $X_s^{U,D}=X_{s^{-}}^{U,D}+\bigtriangleup
 X_s +\bigtriangleup U_s -\bigtriangleup D_s$.
We now take expectation to obtain
\begin{multline*}
\E_x u(X^{U,D}_{T}) -u(x-d_0+u_0)+ \E_x  \left(\int_0^Tc_\delta(X^{U,D}_{t^-})dt +q_u(U_{T}-u_0) +q_d  (D_{T}-d_0) \right) \\
\geq \E_x\int_0^T\epsilon u(X^{U,D}_{t^-})\,dt.
\end{multline*}
Replacing $X_t^{U,D}$ by $X_{t^-}^{U,D}$ in the Lebesgue integrals,
dividing by $T$ and taking limits as $T \rightarrow \infty$, 
we conclude the proof for the case \rm(i) of unbounded variation, in view of \eqref{eq:allcontroled}. 
The case \rm(ii) of bounded variation follows similarly applying 
the change of variables for finite variation processes \cite[Thm. II.31]{ProtterPE}
instead of 
It\^o's formula (see \cite[Thm. II.32]{ProtterPE}).
\end{proof}

The candidate $u$ proposed in the following sections is linear outside an interval and bounded from below. 
This property allows to prove that the reflection strategies are optimal.

\begin{theorem}\label{T:VerificationsBarriersOptimal}
Under the hypotheses of Theorem \ref{T:Verification}, 
if there is a pair of thresholds $a<0< b$ such that $u$ also satisfies
\begin{equation}\label{eq:freeboundary}
\left\{
\aligned
\mathcal{L}u(x)- \epsilon u(x)+c_{\delta}(x)&=0,    &\text{for all  $x \in (a,b)$},\\
u(x)&=u(a)+(a-x)q_u,                      &\text{for all $x \leq a,$}\\
u(x)&=u(b)+(x-b)q_d,                                 &\text{for all  $x \geq b$},  
\endaligned
\right.
\end{equation}
then,
\begin{multline*}
\limsup_{T \to \infty}\frac{1}{T} \E_{x} \Bigg(\int_{(0,T]} c_{\delta}(X_s^{a,b})\,ds+q_uU^{a,b}_T+q_dD^{a,b}_T \Bigg) = \\  
\limsup_{T \to \infty} \frac{\epsilon}{T}\E_{x} \bigg(\int_{(0,T]} u(X^{a,b}_s)\,ds \bigg).
\end{multline*}
\end{theorem}
\begin{remark}
		If the function $u$ satisfies the conditions of  Theorem \ref{T:VerificationsBarriersOptimal} it is bounded from below. 
		If an admissible control $(U,D)$ fails to satisfy condition \eqref{eq:allcontroled}, then $J(x,U;D)=\infty$. 
		To see this, note that when \eqref{eq:allcontroled} does not hold, the $\liminf$ must be strictly greater than zero. 
		On the other hand, by {\rm(ii)} in Definition \ref{D:costfunction}, we have $c(x) \geq K_1\vert u\vert (x)-K_2$ for some positive constants $K_1,K_2$ and we can use exactly the same reasoning as in \cite[Thm. 2.10]{CMO}. 
\end{remark}
\begin{proof} In \eqref{useful}, as $X^{a,b}_{s-}\in(a,b)$, we have
$$
\int_{(0,T]}\mathcal{L}u(X^{a,b}_{t^-})\,dt=\int_{(0,T]}(\epsilon u-c_\delta)(X^{a,b}_{t^-})\,dt.
$$
Furthermore, we have $dU^{a,b}_{t}=dD^{a,b}_{t}=0$ when $X^{a,b}_{t}\in(a,b)$ 
and $u'(x)=-q_u$ for $x\leq a$ and $u'(b)=q_d$ for $x\geq b$. So, using $X_t^{a,b}-X_{t^-}^{a,b}=0$ in the support of $(U^{a,b})^c,(D^{a,b})^c$:
$$
\int_{(0,T]}u'(X^{a,b}_{t^-})(d(U^{a,b})^c_t-d(D^{a,b})^c_t)=-q_u d(U^{a,b})^c_{T}-q_d d(D^{a,b})^c_{T}.
$$
On the other hand $u(X_s^{a,b})- u(X_{s^-}^{a,b}+\bigtriangleup X_s)=-q_u \bigtriangleup U_s-q_d \bigtriangleup D_s $ for all $s >0$.
So we have an equality in \eqref{useful}. %
Taking limsup we obtain the result, because $u(X^{a,b}_T)$ is a bounded quantity due to $u$ being continuous.
\end{proof}

We proceed to show, in the following two theorems, 
that if in the discounted problem there is a candidate function
under the hypothesis of Theorem \ref{T:Verification}, it is the value function of the discounted problem.
\begin{theorem}\label{T:Verificationdiscounted}
Let $\epsilon>0$ and  $\delta>0$, suppose that $c$ is a cost function and there is a convex function $u$ under the same hypothesis as Theorem \ref{T:Verification}.
Then we have:
\[u(x) \leq    \E_{x} \left(  \int_{(0,\infty)} e^{-\epsilon s}  \big(c_{\delta}(X^{U,D}_s)ds +q_u dU_s +q_d dD_s \big) +q_uu_0+q_dd_0 \right)    \] 
for all controlled processes $X^{U,D}$ that satisfy 
\begin{equation}\label{eq:Conditionforverification}
    \E_{x}\int_0^\infty e^{-\epsilon s}u(X^{U,D}_s)ds <\infty.
\end{equation}
\end{theorem}
\begin{proof}
Observe that we can assume
$ \E_x  \int_{(0,\infty)}e^{-\epsilon s}c_{\delta}(X_s^{U,D}) ds   <\infty$,  
otherwise the claim is trivial. 
Let $e(\epsilon)$ be an exponential random variable  with parameter $\epsilon$, 
independent of $X$. We now apply It\^o's formula on the random interval $[0,e(\epsilon)]$, which, with the same arguments as in the proof of Theorem \ref{T:Verification}, gives 
\begin{multline}\label{E:Verificationdiscounted2}
\E_{x}\left(u(X^{U,D}_{e(\epsilon)})-u(x-d_0+u_0)\right)  \\ \geq \E_{x} \left(\int_{(0
,e(\epsilon))}(\epsilon u-c_{\delta})(X_{s^-}^{U,D})ds 
-q_u dU_{e(\epsilon)}-q_d  dD_{e(\epsilon)} \right) . 
\end{multline}
On the other hand, for every stochastic process $Y_s$ with finite variation independent of $e(\epsilon)$ and $g$ continuous satisfying:
$$\E_x \bigg(  \int_{(0,\infty)} e^{-\epsilon s} \left\vert g(X_{s^-}^{U,D})\right\vert dY_s \bigg)<\infty,$$
we have
\begin{align*}
\E_{x}\int_{(0,e(\epsilon)]}g(X_{s^{-}}^{U,D}) dY_s&=\E_{x} \int_{(0,\infty)} \epsilon e^{-\epsilon u} \int_{(0,u]} g(X_{s^-}^{U,D}) \ dY_s du  \nonumber 
\\ &= \E_{x}  \int_{(0,\infty)}\int_{[s,\infty)}\epsilon e^{-\epsilon u}g(X_{s^-}^{U,D}) \ du   dY_s \nonumber \\ &=\E_x  \int_{(0,\infty)} e^{-\epsilon s} g(X_{s^-}^{U,D})dY_s .
 \end{align*}
Therefore, the inequality \eqref{E:Verificationdiscounted2} is equivalent to:
\begin{align*}
0 \geq   \E_{x}  \int_{(0,\infty)}e^{-\epsilon s} \left(-q_u dU_s -q_d dD_s -c_{\delta}(X_{s^-}^{U,D})ds \right) + u(x-d_0+u_0).
\end{align*}
Due to the fact that $X,U,D$ are càdlàg, again, we rewrite the inequality as
\begin{align*}
0 \geq   \E_{x}  \int_{(0,\infty)}e^{-\epsilon s} \left(-q_u dU_s -q_d dD_s -c_{\delta}(X_{s}^{U,D})ds \right) + u(x-d_0+u_0).
\end{align*}
Thus, to finish the proof it is enough to prove $$u(x)-q_uu_0-q_dd_0 \leq u(x-d_0+u_0),$$
which is clearly true because $-q_u \leq u' \leq q_d$, concluding the proof of the Theorem.
\end{proof}
Finally, we formulate the discounted version of Theorem \ref{T:VerificationsBarriersOptimal}. 
We omit the proof because it is exactly the same as in the ergodic version except in the case 
$x\in(a,b)^c$, where the linearity of $u(x)$ must be used. %
\begin{theorem}\label{T:VerificationsBarriersOptimalDISCOUNT}
Assume that $u$ satisfies the hypothesis of Theorem \ref{T:Verificationdiscounted}.
If there is a couple $a<b$ such that $u$ also satisfies \eqref{eq:freeboundary}, then,
\begin{equation*}
\E_{x}   \int_{(0,\infty)}e^{-\epsilon s}\left(c_{\delta}(X_s^{a,b}) ds+ q_udU^{a,b}_s+q_d dD^{a,b}_s  \right)+q_uu_0^{a,b}+q_dd_0^{a,b}   =  u(x).
\end{equation*}
\end{theorem}
\section{The associated Dynkin game}\label{S2:DynkinAdjoint}
Reduction of control problems to optimal stopping problems is a possible solving strategy, 
an initial example being the proposal by Karatzas in \cite{KARATZAS}. 
In this example, a singular control problem that satisfies a verification theorem similar to Theorem \ref{T:Verification} is solved by finding the solution of an associated optimal stopping problem. 
In our situation, due to the nature of our two sided problem, the associated problem turns out to be a Dynkin game. %
The relationship between singular control problems and Dynkin games has been known for a certain time. It was proved in \cite{KW} for finite horizon problems, in \cite{BOETIUS} for diffusions and in \cite{GT} for a more general setting. In fact, as an alternative path to show the relationship between the control problems, the results of this last article could be adapted for Lévy processes if we assume further restrictions on the cost function (see \cite[eq. (13)]{GT0}). 
Apart from the aforementioned articles, our main references are \cite{EG, PG} from which we also borrow some notation.

\subsection{Introduction and properties}
In this section, we work with the function $c_{\delta}$ appearing in Definition \ref{D:costfunction}. 
To ease the notation, we omit the $\delta$, denoting it by $c$. 
The underlying process $X$ is the L\'evy process with finite expectation defined in \eqref{D:LEVYPROCESSEQUATION}, 
$\epsilon$ is a fixed positive real number.
Given the discounted control problem with value function in Definition \ref{D:discountedgame},
the auxiliary problem will be a discounted Dynkin game with an integral cost, 
that we formulate through  a three-dimensional process.
Consider then  $\lbrace Z_t  = r+t \rbrace_{t \geq 0}$ and  $\big\lbrace I_t  =w+ \int_{0}^{t} e^{-\epsilon Z_s} c'(X_s) ds \big\rbrace_{t \geq 0} $, and the process
$$
\X=\{\X_t=(X_t,I_t,Z_t)\}_{t\geq 0},\quad \X_0=\x=(x,w,r).
$$
The process $\X$  is strong Markov with respect to the original filtration $\mathbf{F}=\{\mathcal{F}_t\}_{t \geq 0}$.
As usual, denote by $\mathbf{P}_\x$ and  $\mathbf{E}_\x$ the probability measure and expected value respectively when the process starts at the point $\x$. 
For simplicity, we denote $\mathbf{P}_{x,0,0}$ as $\mathbf{P}_x$ and $\mathbf{E}_{x,0,0}$ as $\mathbf{E}_x$. 
A stopping time is a measurable function $\tau \colon \Omega \rightarrow [0,\infty]$ s.t.
\[ 
\lbrace \omega\colon \tau(\omega)\leq t \rbrace \in \mathcal{F}_t , \quad\text{for all $t\geq 0$}.
\]
We denote the set of stopping times as $\Re$. In our setting, the first entry time to a Borel set is a stopping time. 
As usual, the infimum of the empty set is $\infty$.

Consider now the real functions 
\begin{align*}
G_1(x,w,r)&=w-q_u e^{-\epsilon r} ,\\ 
G_2(x,w,r)&=w+q_d e^{-\epsilon r},
\end{align*}
and denote the payoff
\begin{equation*}
M_\x(\tau,\sigma)=
\E_\x\left(G_1(\X_{\tau})\mathbf{1}_{\{\tau < \sigma\}}  
+ 
G_2(\X_{\sigma})\mathbf{1}_{\{\tau \geq \sigma\}} \right). 
\end{equation*}
Observe that, with the auxiliary function
$$
Q_r(\tau,\sigma)=e^{-r\epsilon}\left(
q_d\mathbf{1}_{\{\tau \geq \sigma\}}e^{-\epsilon\sigma}  
-q_u\mathbf{1}_{\{\tau < \sigma\}}e^{-\epsilon\tau} 
\right),
$$
the payoff can be written as
\begin{equation*}
M_\x(\tau,\sigma)=
\E_\x\left(I(\tau\wedge\sigma)+Q_r(\tau,\sigma)\right),
\end{equation*}
where $Q_r$ depends only on the stopping times, not on the process $X$.
As usual, we denote $Q_0$ by $Q$.
Again, for simplicity, we denote $M_{x,0,0}$ as $M_x$.
Define the value functions 
\begin{equation*}
V_\ast(\x)  = \sup_{\tau} \inf_{\sigma} M_\x(\tau, \sigma),  
\quad  
V^{\ast}(\x)= \inf_{\sigma} \sup_{\tau} M_\x(\tau, \sigma).
\end{equation*} 
In this section,  hitting times of translated sets play a crucial role.
If $\gamma$ is defined by
\[ 
\gamma = \inf \lbrace t\geq 0\colon X_t \in A \rbrace,
\]
then the stopping time $\gamma_y$ is 
\[ 
\gamma_y= \inf \lbrace t\geq 0\colon X_t \in A -y\rbrace.
\]
As L\'evy processes are space invariant, 
if $X$ starts from zero, then $\gamma_y$ has the same distribution
as $\gamma$ when the process starts from $y$.
\begin{definition}[$\epsilon$-Dynkin Game]\label{def:DG}
Given the process $\{\X_t=(X_t,I_t,Z_t)\}_{t\geq 0}$, 
the functions $G_1,G_2$,
$M_\x$, 
and the payoffs $V_\ast,V^{\ast}$ defined above, 
the $\epsilon$-Dynkin game is the problem consisting in finding two stopping times $(\tau^*,\sigma^*)$
s.t.
\begin{equation}\label{eq:nash}
M_\x(\tau,\sigma^{\ast}) \leq M_\x(\tau^{\ast},\sigma^{\ast}) \leq M_\x(\tau^{\ast},\sigma),  \quad\text{for all $\tau, \sigma \in \Re$,} 
\end{equation}
and the \emph{value function} $V_{\epsilon}$ s.t.
\begin{align*}%
V_{\epsilon}(\x)=V_\ast(\x)=V^{\ast}(\x)=M_{\x}(\tau^*,\sigma^*).
\label{eq:dg}
\end{align*}
The stopping times $(\tau^*,\sigma^*)$ in \eqref{eq:nash} constitute a \emph{Nash equilibrium}.
\end{definition}
To ease the notation, we denote $V_\ast(x,0,0),V^{\ast}(x,0,0)$ 
by $V_\ast(x),V^{\ast}(x)$ respectively. The next result is based on \cite[Thm. 2.1]{EG},  
and gives useful properties of the game. 
\begin{proposition}\label{L:Peskir}
The $\epsilon$-Dynkin game of Definition \ref{def:DG} satisfies the following properties:
\vskip2mm\par\noindent
{\rm(i)} The functions $V_\ast$ and $V^{\ast}$ are Borel measurable.
Furthermore, for all $(x,w) \in \mathbb{R}^2$ and $r\geq 0$, it holds  
$$
V_\ast(\x)=V^{\ast}(\x).
$$ 
From now on, we denote $V_{\epsilon}=V_\ast =V^{\ast}$.
\vskip2mm\par\noindent
{\rm(ii)}   Define the sets
\begin{align*}
D_1=\{x\in\R\colon G_1(x,0,0)=V_{\epsilon}(x,0,0)\}, \\ %
D_2=\{x\in\R\colon G_2(x,0,0)=V_{\epsilon}(x,0,0)\}.%
\end{align*}
The stopping times
\begin{equation}\label{eq:tau}
\tau^{\ast}= \inf \lbrace t\geq 0\colon X_t\in D_1\rbrace
\end{equation}
\begin{equation}\label{eq:sigma}
\sigma^{\ast}=\inf \lbrace t\geq 0\colon X_t\in D_2\rbrace,   
\end{equation}
constitute a Nash equilibrium.
\vskip2mm\par\noindent
{\rm(iii)} The following statements hold:
\begin{align}
\E_{x} e^{-\epsilon (\sigma^{\ast}\wedge t)} V_{\epsilon}(X_{\sigma^{\ast}\wedge t})  &\leq  V_{\epsilon}(x)-\E_x\int_0^{\sigma^{\ast}\wedge t}  e^{-\epsilon s} c'(X_s)ds,\label{first}\\
\E_{x}  e^{-\epsilon (\tau^{\ast}\wedge t)} V_{\epsilon}(X_{\tau^{\ast}\wedge t}) &\geq  V_{\epsilon}(x)-\E_x \int_0^{\tau^{\ast}\wedge t}  e^{-\epsilon s} c'(X_s)ds, \label{second}\\
\E_{x}e^{-\epsilon (\tau^{\ast}\wedge \sigma^{\ast}  \wedge t)} V_{\epsilon}(X_{\tau^{\ast}\wedge \sigma^{\ast}  \wedge t})  &
= V_{\epsilon}(x)-\E_x\int_0^{\tau^{\ast} \wedge \sigma^{\ast}\wedge t}  e^{-\epsilon s} c'(X_s)ds.\label{martingale}
\end{align}
\end{proposition}
\begin{proof}[Proof of {\rm(i)}] 
For both statements of (i), we apply \cite[Thm. 2.1]{EG}. To do so, we only have to check $\E_{\x}\sup_t \vert G_i(\mathbf{X}_t) \vert  < \infty, \ i=1,2$. This is clear from the fact
$$\E_{x} \int_{0}^{\infty} \vert c'(X_s) \vert e^{-\epsilon s} ds  < \infty .$$ 
\vskip1mm\par\noindent
\textit{Proof of {\rm(ii)}.}
Again, from \cite[Thm. 2.1]{EG},  the stopping times
\begin{align*}
\tau^{\ast}&= \inf \lbrace u\geq 0\colon G_1(X_u,I_u,Z_u)=V_{\epsilon}(X_u,I_u,Z_u) \rbrace,\\ 
\sigma^{\ast}&=\inf\lbrace u\geq 0\colon G_2(X_u,I_u,Z_u)=V_{\epsilon}(X_u,I_u,Z_u) \rbrace,
\end{align*}
define a Nash equilibrium. In order to prove \eqref{eq:tau} we verify
\begin{equation}
G_1(\x)=V_{\epsilon}(\x) \Leftrightarrow G_1(x,0,0)=V_{\epsilon}(x,0,0).\label{eq:nash1} 
\end{equation}
Observe that
\begin{align}
V_{\epsilon}(x,w,r)&= \sup_{\tau}\inf_{\sigma}\E_{\x}\Bigg(w+\int_{0}^{\tau \wedge \sigma}c'(X_s)e^{-\epsilon Z_s}ds
+q_d\mathbf{1}_{\{\tau \geq \sigma\}}e^{-\epsilon Z_{\sigma} }  
\notag\\&
\qquad\qquad\qquad\quad\qquad\qquad
-q_u\mathbf{1}_{\{\tau<\sigma\}}e^{-\epsilon Z_\tau}   \Bigg)\notag\\
&=w+e^{-r}\sup_{\tau}\inf_{\sigma}\E_{x,0,0}\Bigg(\int_{0}^{\tau \wedge \sigma}c'(X_s)e^{-\epsilon s}ds  
\notag\\&
\qquad\qquad\qquad\qquad\qquad\qquad
+q_d \mathbf{1}_{\{\tau \geq \sigma\}}e^{-{\sigma} \epsilon}  -q_u \mathbf{1}_{\{\tau < \sigma\}}e^{- \epsilon {\tau} } \Bigg)\notag\\
&=w+e^{-r}V_{\epsilon}(x,0,0),\label{eq:ve}
\end{align} 
where we used the fact  that, for every $A \in \mathcal{B}(\mathbb{R}^3)$
\[
\P_{x,w,r}((X_t,I_t,Z_t) \in A )= \P_{x,0,0}((X_t,w+I_t,r+t) \in A), \quad \text{ for all } t \geq 0. 
\]
Using \eqref{eq:ve} and
$$
G_1(x,w,r)=w+e^{-r}G_1(x,0,0).
$$
the equation \eqref{eq:nash1} follows. 
The proof of \eqref{eq:sigma} is analogous.
\vskip1mm\par\noindent
\textit{Proof of {\rm(iii)}.}
First consider \eqref{first}. 
We apply (E) of \cite[Thm. 2.1]{PG}, and use 
\eqref{eq:ve}:
\begin{align*}
V_{\epsilon}(x)    &\geq \E_xV_{\epsilon}(X_{t\wedge \sigma^{\ast}},I_{t\wedge \sigma^{\ast}}, Z_{t \wedge \sigma^{\ast}}) \\
        &= \E_x\left(
        e^{-\epsilon(t \wedge \sigma^{\ast})}V_{\epsilon}(X_{t\wedge \sigma^{\ast}},0,0)+I_{t\wedge \sigma^{\ast}}
        \right),
\end{align*}
that is \eqref{first} (in fact the supermartingale property holds). 
Inequality \eqref{second} is proved analogously, and \eqref{martingale}
follows from the respective super and sub martingale properties in \eqref{first} and \eqref{second}.\end{proof}
%
\subsection{The value function}
We now return to our one-dimensional setting. With a slight abuse of notation, consider the functions
$$
\aligned
V_{\epsilon}(x)&=V_{\epsilon}(x,0,0),\\
G_1(x)&=G_1(x,0,0)=-q_u,\\
G_2(x)&=G_2(x,0,0)=q_d.\\
\endaligned
$$
The optimal stopping times defined in \eqref{eq:tau} and \eqref{eq:sigma} that constitute the Nash equilibrium are then
$$
\aligned
\tau^*&=\inf\{t\geq 0\colon V_{\epsilon}(X_t)=-q_u\},\\
\sigma^*&=\inf\{t\geq 0\colon V_{\epsilon}(X_t)=q_d\}.
\endaligned
$$
Observe also that
$$
-q_u\leq V_{\epsilon}\leq q_d.
$$
We proceed to study properties of the value function $V_{\epsilon}$ of the $\epsilon$-Dynkin game,
that will allow us to construct a candidate for the 
verification theorems of Section \ref{S1:Verification}.
The next proposition is crucial to deduce the necessary properties of $V_{\epsilon}$ without giving an explicit solution. 
\begin{proposition}\label{P:AdjoinDynkinVISINCREASING}
The value function $V_{\epsilon}$ is non-decreasing.
\end{proposition}
\begin{proof}
Take $x<y$. Observe first that, as $Q$ does not depend on $x$,
$$
\E_x Q(\tau^{\ast},\sigma^{\ast }_{y-x})
=\E_y Q(\tau^{\ast}_{x-y},\sigma^{\ast })=
\E Q(\tau^{\ast}_x,\sigma^{\ast }_{y}).
 $$
Then,
\begin{multline*}
V_{\epsilon}(x)-V_{\epsilon}(y)   \leq  M_x(\tau^*,\sigma^*_{y-x})-M_y(\tau^*_{x-y},\sigma^*)\\
=\E_x \left(I(\tau^{\ast} \wedge \sigma^{\ast }_{y-x})+ Q(\tau^{\ast},\sigma^{\ast }_{y-x})\right)
-\E_y\left(I(\tau^{\ast}_{x-y},\sigma^{\ast })+Q(\tau^{\ast}_{x-y},\sigma^{\ast })\right) \\
=\E \int_0^{\tau^{\ast}_x \wedge \sigma^{\ast}_y}e^{-\epsilon s} \big(c'(x+X_s)-c'(y+X_s)\big)ds  \leq 0,
\end{multline*}
by the convexity of $c$, concluding the proof.
\end{proof}
\begin{proposition}\label{C:AdjoinDynkinVISINCREASING}
Define
$$
a^*_{\epsilon}=\sup\lbrace x\colon V_{\epsilon}(x)=-q_u\rbrace,
\qquad
b^*_{\epsilon}=\inf\lbrace x\colon V_{\epsilon}(x)=q_d\rbrace 
$$
that could take the values $a^*_{\epsilon}=-\infty, \ b^*_{\epsilon}=\infty$.
Then,
\vskip2mm\par\noindent
{\rm(i)} the stopping times \eqref{eq:tau} and \eqref{eq:sigma} satisfy:
$$
\aligned
\tau^*&=\inf\{t\geq 0\colon X_t\leq a^*_{\epsilon}\}=\inf\{t\geq 0\colon V_{\epsilon}(X_t)=-q_u\},\\
\sigma^*&=\inf\{t\geq 0\colon X_t\geq b^*_{\epsilon}\}=\inf\{t\geq 0\colon V_{\epsilon}(X_t)=q_d\}.
\endaligned
$$
In other words, each one of the sets $D_1, \ D_2 $ is a half line, possibly empty, and the value function satisfies
\begin{equation}\label{E:constant}
V_{\epsilon}(x)=
\begin{cases}
-q_u,&\text{for $x\leq a^*_{\epsilon},$}\\
\ \ q_d,&\text{for $x\geq b^*_{\epsilon}.$} 
\end{cases}
\end{equation}
\vskip2mm\par\noindent
{\rm(ii)} We have strict inequalities $a^*_{\epsilon}<0<b^*_{\epsilon}$.
\vskip2mm\par\noindent
{\rm(iii)} Furthermore, for $\epsilon$ small enough in the case $\alpha=0$, 
or in the case $\alpha>0$, the thresholds
$a^*_{\epsilon}$ and $b^*_{\epsilon}$ are finite.
\end{proposition}
\begin{proof}[Proof of {\rm(i)}]
The claim follows from the fact that $V_{\epsilon}$ is increasing, $-q_u \leq V_{\epsilon}(x) \leq q_d \ \text{ for all } x \in  \mathbb{R}$ and the definition of $\tau^{\ast}$ and $\sigma^{\ast}$.  
\vskip1mm\par\noindent
\textit{Proof of {\rm(ii)}}
The fact that $a_{\epsilon}^{\ast} \leq 0 \leq b^{\ast}_{\epsilon}$ and $a_{\epsilon}^{\ast} \neq b^{\ast}_{\epsilon}$ is deduced from the fact that $c'(x)$ is increasing, $c'(0)=0$ and $G_1 < G_2$.  

Now, assume $\P(\tau^{\ast}>0)=1$ and observe for all $\sigma \in \Re $:
\begin{align*}
M(\tau^{\ast},\sigma) 
&=\E \left(\int_0^{\tau^{\ast}\wedge \sigma }c'(X_s)e^{-\epsilon s}ds -q_u \mathbf{1}_{\lbrace \tau^\ast < \sigma \rbrace}e^{-\epsilon \tau^\ast}+q_d \mathbf{1}_{\lbrace \tau^\ast \geq \sigma \rbrace}e^{-\epsilon \sigma}  \right) \nonumber \\
&\leq \E\left( \int_0^{\sigma} \vert c'(X_s)\vert e^{-\epsilon s}ds +q_d \mathbf{1}_{\lbrace \tau^\ast \geq \sigma \rbrace}e^{-\epsilon \sigma}\right) \nonumber \\
&= \E\left( \int_0^{\sigma} \vert c'(X_s)\vert e^{-\epsilon s}ds +q_de^{-\epsilon \sigma}\right).
\end{align*}
Define $\sigma=\inf\lbrace t \geq 0\colon \vert c'(X_t) \vert \geq q_d \epsilon \rbrace $, 
which is strictly positive almost surely because $c'(0)=0$ and $X$ is right continuous. 
For this choice of $\sigma$, the expectation $\E \int_0^{\sigma}\left( \vert c'(X_s)\vert-q_d \epsilon \right)ds <0$, so $M(\tau^{\ast},\sigma^{\ast})\leq M(\tau^{\ast},\sigma)<q_d$. This implies that $\sigma^{\ast}$ is not identically zero, and consequently $b^{\ast}_{\epsilon}>0$. A similar argument shows that $a^{\ast}_{\epsilon}<0$.
Statement (iii) is a consequence of the following  Proposition \ref{P:BOUNDSAB}.
\end{proof} 
The following proposition provides bounds for $(a^{\ast}_{\epsilon},b^{\ast}_{\epsilon})$,
that, despite not being optimal, 
will provide 
bounds for the optimal barriers of a Dynkin game with parameter $\overline{\epsilon}$ such that $0<\overline{\epsilon} \leq \epsilon$, as will be seen below.
Define now, for $\ell>0$, the stopping time
$\gamma^{\ell}$ as
\begin{equation*}
\gamma^{\ell}=\inf \left\lbrace t\geq 0\colon  \left\vert X_t-X_0 \right\vert \geq \frac{\ell}{2} \right\rbrace.     
\end{equation*}
\begin{proposition}\label{P:BOUNDSAB}
  With the notation of Definition \ref{D:costfunction}, if $\alpha>0$ or if $\epsilon$ is small enough in case $\alpha=0$, there is a constant $L$ such that the optimal thresholds of the $\overline{\epsilon} $-Dynkin game  $(a^{\ast}_{\overline{\epsilon} },b^{\ast }_{\overline{\epsilon} }) \subset[-L,L]$  for every $\overline{\epsilon}  \leq \epsilon$. %
\end{proposition}
\begin{proof}
First, we assume $\alpha=0$. From Definition \ref{D:costfunction}, for $\ell$ large enough, for all $x\geq \ell/2$ we have 
$
c'(x)\geq N/2.
$
Then for all $\overline{\epsilon} >0$:
\begin{equation*} 
\E_{\ell} \int_0^{\gamma^\ell}c'(X_s)e^{-\overline{\epsilon}  s}ds 
\geq \frac{N}{2}\E_{\ell} \left( \frac{1-e^{-\overline{\epsilon}  \gamma^{\ell}}}{\overline{\epsilon} } \right).
\end{equation*}
On the other hand,
$$
\E_{\ell}\gamma^\ell=\E \inf\left\lbrace t\geq 0\colon |X_t|\geq \frac{\ell}{2} \right\rbrace \to\infty\text{ as $\ell\to\infty$},
$$
we find an $\ell$ s.t.
\begin{equation}\label{E:boundDynkin1}
\E_{\ell} \gamma^{\ell}  >\frac{2}{N}(q_u+q_d).
\end{equation}
Now, for each fixed $\ell$, the expression $\frac{1-e^{-\overline{\epsilon}  \gamma^{\ell}}}{\overline{\epsilon} } $ is non-negative and bounded by $\gamma^{\ell}$, because $\gamma^\ell>0$. 
Therefore, for each fixed $\ell$ satisfying \eqref{E:boundDynkin1}, using dominated convergence and the fact $\E_{\ell}(\gamma^{\ell})< \infty$ when $X$ is not the null process, we have that 
$$ 
\E_{\ell} \left( \frac{1-e^{-\overline{\epsilon}  \gamma^{\ell}}}{\overline{\epsilon} } \right) \nearrow  \E_{ \ell}\gamma^{\ell} \ \text{ as} \ \overline{\epsilon}  \to 0 .
$$
Thus, using \eqref{E:boundDynkin1} we can take $\epsilon$ small enough such that for every $\overline{\epsilon}  \leq \epsilon$ we have
\begin{equation}\label{E:boundDynkin0}
\E_{\ell} \int_0^{\gamma^\ell}c'(X_s)e^{-\overline{\epsilon}  s}ds 
\geq q_u+q_d.
\end{equation}
We now take $L:=3 \ell/2 $ and prove $b^{\ast}_{\overline{\epsilon} }\leq L$ for every 
$\overline{\epsilon}  \leq \epsilon$. 
Assume, by contradiction, that $b^{\ast}_{\overline{\epsilon} }>L$, what implies $q_d>V_{\overline{\epsilon}}(\ell)$ (the function $V_{\overline{\epsilon}}$ is the value of the $\overline{\epsilon} $-Dynkin game) of Definition \ref{def:DG}.  
Now, by \eqref{martingale},
and using that $V_{\overline{\epsilon}}(x)\geq -q_u$ we have 
\begin{multline}\label{E:boundDynkin2}
q_d > V_{\overline{\epsilon}}(\ell)
=\E_{\ell} \left(  \int_0^{\gamma^\ell}c'(X_s)e^{-\overline{\epsilon}  s}ds   
+e^{-\overline{\epsilon}  \gamma^\ell}V_{\overline{\epsilon}}(X_{\gamma^\ell})\right)\\
\geq  \E_{\ell}    \int_0^{\gamma^\ell}  c'(X_s)e^{-\overline{\epsilon}  s}ds  -q_u \geq q_d.
\end{multline}
by \eqref{E:boundDynkin0}, what is a contradiction. 
The other bound is analogous (thus $L$ must be taken as the one with the greater absolute value), concluding the proof of the first statement when $\alpha=0$. 

Assume now $\alpha>0$. We follow a similar approach. Instead of taking $\gamma^{\ell}$ we consider $\gamma^1$. Observe that if $\ell \geq$ 1 we have for every $\overline{\epsilon}\leq \epsilon$:
\begin{multline}\label{E:boundDynkin3} 
\E_{\ell} \int_0^{\gamma^1}c'(X_s)e^{-\overline{\epsilon}s}ds 
\geq c'(\ell-1)\E_{\ell} \left( \frac{1-e^{-\overline{\epsilon}  \gamma^{1}}}{\overline{\epsilon} } \right) 
\\ \geq 
c'(\ell-1)
\E \left( \frac{1-e^{-\epsilon  \gamma^{1}}}{\epsilon } \right) \geq q_u+q_d,
\end{multline}
for $\ell$ big enough. and we take $\gamma^1$ instead of $\gamma^{\ell}$, equation \eqref{E:boundDynkin0} is satisfied (with $\ell$ as starting point). Now, if take $L:= 3\ell/2$ again, we arrive to the same contradiction as in equation \eqref{E:boundDynkin2}.

\end{proof}

From now on, we assume $\epsilon$ is small enough in case $\alpha=0$ or $\alpha>0$ so that Proposition \ref{P:BOUNDSAB} holds.

\begin{proposition}\label{P:DynkinLipschitz}
The value function $V_{\epsilon}$ is Lipschitz continuous.
\end{proposition}
\begin{proof}
In view of \eqref{E:constant} we need to consider the pairs $x,y$ s.t. $a^*_{\epsilon}\leq x<y\leq b^*_{\epsilon}$. 
As in the proof of Proposition \ref{P:AdjoinDynkinVISINCREASING}, we have,
\begin{multline*}
0 \leq V_{\epsilon}(y)-V_{\epsilon}(x) \leq M_y(\tau^{\ast},\sigma^{\ast}_{x-y})- M_x(\tau^{\ast}_{y-x},\sigma^{\ast} )= \\ 
\E \left( \int_0^{ \tau^{\ast}_{y} \wedge \sigma^{\ast}_{x}} e^{-\epsilon s}( c'(X_{s} +y)-c'(X_{s} +x) ) ds \right)
\\
\leq 
 \sup_{z \in [a^{\ast}-b^{\ast},b^{\ast}-a^{\ast}]} \vert c''(z) \vert    \vert y-x \vert \ \E ( \tau^{\ast}_{y} \wedge \sigma^{\ast}_{x}) .
\end{multline*}
We conclude the proof using the fact that 
$\tau^{\ast}_{y} \wedge \sigma^{\ast}_{x}$ is bounded by
the stopping time
\[ \inf \lbrace t\geq 0\colon X_t \notin (a^{\ast}_{\epsilon}-b^{\ast}_{\epsilon},b^{\ast}_{\epsilon}-a^{\ast}_{\epsilon})\rbrace , 
\]
that has finite expectation.
\end{proof}
When the process $X$ has unbounded variation, it is possible to prove the continuous differentiability of the value function $V_{\epsilon}$. 
This will be necessary, in this situation, to construct a candidate that verifies condition (i) of Theorem \ref{T:Verification}.

\begin{proposition}\label{L:DynkinC1indense}
If the L\'evy process $X$ has unbounded variation, the function $V_{\epsilon}$ is differentiable, and
$$
V'_{\epsilon}(x)=\E\int_0^{\tau^*_x\wedge\sigma^*_x}c''(x+X_s)e^{-\epsilon s}\,ds.
$$
Furthermore, $V'_{\epsilon}$ is continuous.%
\end{proposition}
\begin{proof} Take $x\in\R$ and $h>0$. We obtain the bound
$$
\aligned
V_{\epsilon}(x+h)  &\leq \E_{x+h}\left(\int_0^{\tau^*\wedge\sigma^*_{-h}}c'(X_s)e^{-\epsilon s}\,ds+Q(\tau^*,\sigma^*_{-h})\right)\\
        &=\E\left(\int_0^{\tau^*_{x+h}\wedge\sigma^*_{x}}c'(x+X_s+h)e^{-\epsilon s}\,ds+Q(\tau^*_{x+h},\sigma^*_{x})\right).\\
\endaligned
$$
Similarly
$$
\aligned
V_{\epsilon}(x)  &\geq \E_{x}\left(\int_0^{\tau^*_h\wedge\sigma^*}c'(X_s)e^{-\epsilon s}\,ds+Q(\tau^*_h,\sigma^*)\right)\\
        &=\E\left(\int_0^{\tau^*_{x+h}\wedge\sigma^*_{x}}c'(x+X_s)e^{-\epsilon s}\,ds+Q(\tau^*_{x+h},\sigma^*_{x})\right).\\
\endaligned
$$
Subtracting and applying the mean value theorem, we get
$$
{V_{\epsilon}(x+h)-V_{\epsilon}(x)\over h}\leq \E \int_0^{\tau^*_{x+h}\wedge\sigma^*_{x}}c''(x+X_s+\theta h)e^{-\epsilon s}\,ds ,
$$
where $0\leq\theta\leq 1$. 
Furthermore, If we assume  
\begin{equation}\label{E:Contdifferentiablelimittimes}
   \lim_{h \to 0} \E \vert \tau^*_{x+h} \wedge \sigma^{\ast}_x - \tau^{\ast}_x \wedge \sigma^{\ast}_x  \vert =0. 
\end{equation}    
Then using the fact that
$\E(\tau^*\wedge\sigma^*)<\infty$,
$c''(x)$ is continuous and bounded in compacts, by dominated convergence
$$
\limsup_{h\downarrow0}{V_{\epsilon}(x+h)-V_{\epsilon}(x)\over h}\leq \E\int_0^{\tau^*_{x}\wedge\sigma^*_{x}}c''(x+X_s)e^{-\epsilon s}\,ds .
$$
To prove \eqref{E:Contdifferentiablelimittimes}, we analyze when
\begin{equation}\label{E:tauconvergence}
\tau^{\ast}_{x+h}\wedge \sigma^{\ast}_x  \downarrow\tau^{\ast}_{x}\wedge \sigma^{\ast}_x  , \ a.s.\ \text{ as $h\downarrow 0$.}
\end{equation}
First, observe that $\tau^{\ast}_{x+h}$ decreases when $h$ decreases.
Now, for $\alpha>0$, denote $I_\alpha=\inf_{0\leq t\leq \alpha}X_t$.
By the strong Markov property, we have
$$
\P(\tau^*_{x+h} -\tau^*_{x}>\alpha)
        \leq\P (I_{\alpha} >-h ),  
$$
then
$$
\lim_{h \downarrow 0} \P(\tau^*_{x+h} -\tau^*_{x}>\alpha)
        \leq \lim_{h \downarrow 0} \P (I_{\alpha} >-h )=\P(I_{\alpha}=0).  
$$

We deduce that \eqref{E:tauconvergence} holds when the random variable $I_\alpha$ has no atoms,
i.e. if and only if 0 is regular for $(-\infty,0)$.
This is the situation when the process has unbounded variation (see \cite[Prop. 7]{alili}).
 We conclude that \eqref{E:Contdifferentiablelimittimes} holds using \eqref{E:tauconvergence} and the fact that, for $h$ small enough, the times in the sequence are dominated by the first exit time of the set $(a^{\ast}-b^{\ast},b^{\ast}-a^{\ast}) $.

Similar arguments apply in the other 3 cases: when $h>0$ to obtain lower bounds, and the other two situations for $h<0$.  The continuity follows taking limits under dominated convergence, with similar arguments as above. 
\end{proof}
\section{Construction of optimal controls and proof of main results}\label{S3:Candidate}
In this section, we first give some preliminary results and then provide the proofs of Theorems \ref{T:DISCOUNTEDPROBLEMSOLUTION} and \ref{T:ErgodicProblemsolution}.

\subsection{Properties of the primitive of the value function}
Using properties of the value function $V_{\epsilon}$ of the $\epsilon$-Dynkin game introduced in the previous section, 
we study properties of the indefinite integral or primitive function, i.e. a function $W$ such that $W'=V_{\epsilon}$. 
As a direct consequence of Proposition  \ref{P:technicalinfinitesimal} 
(in both cases, with and without bounded variation) follows the next result.

\begin{proposition}\label{P:DomainInfinitesimalgenerator}
If $W$ is a primitive function of the value  $V_{\epsilon}$  of the $\epsilon$-Dynkin game in Definition \ref{def:DG}, 
then $W$ is in the domain of the infinitesimal generator $\mathcal{L}$, 
and the function 
$x \rightarrow \mathcal{L}W(x)$ is continuous.
\end{proposition}
%
\begin{lemma}\label{L:HarmonicityPrimitiveDynkin}
Let $V_{\epsilon}$ be the value of the $\epsilon$-Dynkin game of Definition \ref{def:DG}
and $W$ a primitive of $V_{\epsilon}$. Then, the function 
\[
x \rightarrow \mathcal{L}W(x)+c(x)- \epsilon W(x), 
\] 
\vskip2mm\par\noindent
{\rm(i)}
is constant in $(a^{\ast}_{\epsilon},b^{\ast}_{\epsilon})$,
\vskip2mm\par\noindent
{\rm(ii)}
increases in $(b^{\ast}_{\epsilon},\infty)$,
\vskip2mm\par\noindent
{\rm(iii)} decreases in $(-\infty, a^{\ast}_{\epsilon}).$
\end{lemma}
\begin{proof}
We begin by proving  \rm{(iii)}.  For $r>0,$ let $\eta_r$  be the stopping time defined as $$\eta_r = \inf \lbrace t \geq 0 \colon X_t \notin (-r,r) \rbrace .$$ 
For $h>0$ we need to prove that the difference 
\begin{align*}
\lim_{r \to 0^+} \frac{\displaystyle \E \bigg(e^{-\epsilon \eta_r}W(X_{\eta_r}+x+h)-W(x+h)+\int_0^{\eta_r} c(X_s+x+h)e^{-\epsilon s}ds\bigg)}{\displaystyle \E\eta_r}  \nonumber \\ 
- \lim_{r \to 0^+}  \frac{\displaystyle \E \bigg(e^{-\epsilon \eta_r }W(X_{\eta_r }+x)-W(x)+\int_0^{ \eta_r } c(X_s+x)e^{-\epsilon s}ds \bigg)}{\E \eta_r }
\end{align*}
is equal or smaller than zero (see \cite[Thm. V.5.2]{D}). Fix $r>0$  such that $x+h+r<a^{\ast}_{\epsilon}$ and define the differentiable function $H^r$ in a small interval of $x$ (differentiability can be proven using the dominated convergence theorem):
\begin{equation*}
H^{r}(h) \colon  =  
\frac{\displaystyle \E \bigg(e^{-\epsilon \eta_r }W(X_{\eta_r  }+x+h)-W(x+h)+\int_0^{\eta_r   } c(X_s+x+h)e^{-\epsilon s}ds\bigg)}{\displaystyle \E\eta_r}.  
\end{equation*}
Then there exists a $\theta_h^{r} \in [0,1]$ such that $H^{r }(h)-H^{r }(0)$ is equal to

\begin{equation*}
\frac{h}{\E\eta_r} \E\bigg(e^{-\epsilon \eta_r  }V_{\epsilon}(X_{\eta_r  }+x+ \theta_h^{r  } h) -V_{\epsilon}(x+\theta_h^{r } h ) + \int_0^{\eta_r } c'(X_s+x+\theta_h^{r } h)e^{-\epsilon s} ds  \bigg).
\end{equation*}
From \eqref{first} in Proposition \ref{L:Peskir} we know that the above expression
is equal or smaller than zero for every $r<a_{\epsilon}^{\ast}-h-x$, thus concluding the proof of \rm{(iii)}. The proofs of \rm{(i)} and \rm{(ii)}  follow the same line of reasoning using \eqref{second} and \eqref{martingale} instead of \eqref{first}. This concludes the proof of the Lemma.
\end{proof}
\subsection{Continuity properties of reflecting strategies}
To prove  Theorems \ref{T:DISCOUNTEDPROBLEMSOLUTION} and \ref{T:ErgodicProblemsolution} we need results for the reflecting strategies. In this subsection, we assume $b>0$ and $x \in [0,b]$ (implying $d^{0,b}_0=u^{0,b}_0=0$). Moreover, inequality \eqref{eq:reflectingbound} is used when applying integration by parts.  

\begin{proposition}\label{P:reflectingcontinuity}
For every $b>0$, $b_0>0$ and  $r>0,$ the following inequality holds: 
    \begin{multline*}
            \left\vert b_0\E_x D^{0,b_0}_t -b \E_x D^{0,b}_t \right\vert 
            \leq 2^{-1}\vert b^2-b_0^2\vert \\
            +t\int_{(-r,r)}y^2 \Pi(dy) + t \vert b_0-b \vert \left(\E_x \vert X_1 \vert + \sqrt{6}\int_{(-r,r)^c}\vert y \vert \Pi(dy)   \right) ,
                \end{multline*}
for all  $    t>0$. 
\end{proposition}
\begin{proof}
 According to \cite[Thm. 6.3]{AAGP}, for every $b >0$
\begin{equation}\label{eq:assmusen6.3}
 2bD^{0,b}(t)=x^2-(X^{0,b})^2 _{t}+2\int_0^t X^{0,b}_{s^-}dX_s+ [X,X]^c (t)+J^b(t)   ,
\end{equation}
 where $J$ is an increasing and finite process defined by:
 $$J^b(t)= \sum_{0 <s \leq t} \varphi(X^{0,b}_{s^-},\Delta X_s,b), $$
 with 
 \begin{equation*}
\varphi(x,y,b)=\left\{
\aligned
-(x^2+2xy), \qquad    &\text{if $y \leq -x$},  \\
y^2, \qquad                                  &\text{if $-x<y<b-x $,}\\
2y(b-x)-(b-x)^2, \qquad                                 &\text{if $y\geq b-x $,}  
\endaligned
\right.
\end{equation*}
whose domain is the set $\lbrace (x,y,b), \ 0 \leq x \leq b, \ y \in \mathbb{R} \rbrace$.
Moreover, $0 \leq \varphi^b(x,y)\leq y^2$ for all $  (x,y) \in [0,b] \times \mathbb{R}$. Furthermore, due to the fact that $\vert X^{0,b_0}_t-X^{0,b}_t\vert \leq \vert b_0- b\vert $ for all $t \geq 0$, (see \cite[Thm. 2.1]{KRLS}) and using \eqref{eq:assmusen6.3} we get:
\begin{multline*}
      \left\vert b_0  \E_x D^{0,b_0}_t  -b \E_x D^{0,b}_t \right\vert
           \\  \leq 2^{-1}\vert b^2-b_0^2\vert +t\int_{(-r,r)}y^2 \Pi(dy) + t \vert b_0-b \vert \E_x \vert X_1 \vert 
            \\+ 2^{-1}\left\vert\E_x  \sum_{0 <s \leq t, \ \vert\Delta X_s \vert \geq r}\left( \varphi(X^{0,b_0}_{s^-},\Delta X_s,b_0)-\varphi(X^{0,b}_{s^-},\Delta X_s,b) \right) \right\vert ,
                \end{multline*}
for all $t>0$.  Therefore, to finish the proof we need to prove 
\begin{multline}\label{eq:boundforbarphi}
    2^{-1}\left\vert\E_x  \sum_{0 <s \leq t, \ \vert \Delta X_s \vert \geq r}\left( \varphi(X^{0,b_0}_{s^-},\Delta X_s,b_0)-\varphi(X^{0,b}_{s^-},\Delta X_s,b) \right)  \right\vert 
    \\ \leq t \vert b_0-b \vert \sqrt{6}\int_{(-r,r)^c} \vert y \vert \Pi(dy).   
    \end{multline}
For that endeavor, observe that $\varphi$ is continuously differentiable in its domain and  
$$ \vert \vert \nabla \varphi(x,y,b) \vert \vert   \leq 2 \sqrt{3} \vert y \vert. $$
From this inequality and the fact 
$$ \lVert (X^{0,b_0}_{s^-},\Delta X_s,b_0)-(X^{0,b}_{s^-},\Delta X_s,b)   \rVert \leq \sqrt{2 } \vert b_0 -b \vert ,$$ we deduce
\begin{multline*}
    2^{-1}\left\vert\E_x  \sum_{0 <s \leq t, \ \vert \Delta X_s \vert \geq r}\left( \varphi(X^{0,b_0}_{s^-},\Delta X_s,b_0)-\varphi(X^{0,b}_{s^-},\Delta X_s,b) \right)  \right\vert 
    \\  \leq   \sqrt{6} \vert b_0-b\vert   \E_x  \sum_{0 <s \leq t, \ \vert \Delta X_s \vert \geq r} \vert \Delta X_s \vert      . 
    \end{multline*}
Rewriting the expectation in the right term of the inequality, we get
\begin{multline*}
    2^{-1}\left\vert\E_x  \sum_{0 <s \leq t, \ \vert\Delta X_s\vert \geq r}\left( \varphi(X^{0,b_0}_{s^-},\Delta X_s,b_0)-\varphi(X^{0,b}_{s^-},\Delta X_s,b)  \right) \right\vert 
    \\  \leq  \sqrt{6}  t  \vert b_0-b\vert \int_{(-r,r)^c} \vert y \vert \Pi(dy),   
    \end{multline*}
    giving  \eqref{eq:boundforbarphi} and concluding the proof of the proposition.
\end{proof}
\begin{lemma}\label{L:reflectingcontinuity}
For every $b_0>0$
\vskip2mm\par\noindent
{\rm(i)}
$\lim_{b \to b_0}
\E_x \displaystyle{\int_{(0,\infty)}}e^{-\epsilon s}dD_s^{0,b}=\E_x\displaystyle{\int_{(0,\infty)}}e^{-\epsilon s}  dD_s^{0,b_0}$, 
\vskip2mm\par\noindent
{\rm(ii)}
$\lim_{b \to b_0} \lim_{T \to \infty} \frac{1}{T}\E_x D^{0,b}_T= \lim_{T \to \infty} \frac{1}{T} \E_x D^{0,b_0}_T  $. 
\end{lemma}
\begin{proof}
By integration by parts and using \eqref{eq:reflectingbound} we deduce that $\rm{(i)}$ is equivalent to proving that the equality
$$\lim_{b \to b_0 }
\E_x \displaystyle{\int_0^{\infty}}e^{-\epsilon s}  D_s^{0,b}ds   =\E_x \displaystyle{\int_0^{\infty}}e^{-\epsilon s}  D_s^{0,b_0} ds  ,$$
holds. Again, this is clear from Proposition \ref{P:reflectingcontinuity}. The second statement is clearly deduced from the same proposition.
\end{proof}
\begin{lemma}\label{L:reflectingunboundedinfinity}
If the process $ X $ has unbounded variation, then
\vskip2mm\par\noindent
{\rm(i)}
$\lim_{b \to 0}
\E_x \displaystyle{\int_{(0,\infty)}}e^{-\epsilon s}  dD_s^{0,b}=\infty$, 
\vskip2mm\par\noindent
{\rm(ii)}
$\lim_{b \to 0} \lim_{T \to \infty} \frac{1}{T}\E_x D^{0,b}_T= \infty $. 
\end{lemma}
\begin{proof}  In the case where $\nu \neq 0$ statement (i) is deduced by integration by parts, \cite[Thm. 6.3]{AAGP} and the fact that $D_t^{0,b}$ increases in $t$. 
In that case, the statement (ii) follows from \cite[Cor. 6.6]{AAGP}.
 We proceed to study the case $\nu =0$ and $\displaystyle{\int_{(0,1)}}y \Pi(dy)=\infty$. It is clear that for every $t>0$,  $\Delta D^{0,b}_t\geq (\Delta X_t-b )\mathbf{1}_{\lbrace \Delta X_t \geq b \rbrace }$, 
 so
$$ 
\E_x D^{0,b}_T \geq T \int_{(b,\infty)}(y-b)\Pi(dy). 
$$
Therefore, statement (i) is proven by integration by parts and taking limit as $b\to0$. 
To prove (ii),  divide by $T$ and take limit as $b\to0$ in the inequality. 
For the case $\nu =0$ and 
$\displaystyle{\int_{(-1,0)}}y\Pi(dy)=\infty$ we use the same argument for the process $-X$, 
to prove 
$$
\lim_{b \to 0}\lim_{T \to \infty}\frac{1}{T}\E_x U^{0,b}_T=\infty, \qquad \lim_{b \to 0}
\E_x \displaystyle{\int_{(0,\infty)}}e^{-\epsilon s}  dU_s^{0,b}=\infty .
$$
Then we use the fact that $X^{0,b}_t=  X_t-D^{0,b}_t +U^{0,b} _t\in [0,b]$ for all $t$ and conclude the claim.
\end{proof}
\begin{proposition}\label{P:reflectingboundedlimit}
If the process $\lbrace X_t \rbrace$ has bounded variation and non-positive drift $\tilde \mu$ (see \eqref{eq:driftboundedvariation0}) then for all $b>0$  
 $$\vert \E_x D^{0,b}_t - \E_x S^+_t  \vert \leq t \int_{(0,b]} y \Pi(dy)+bt \Pi([b,\infty)) .$$
\end{proposition}
\begin{proof}
Fix $b>0$. On the one hand, observe that the process $D^{0,b}_t$ can be rewritten as:
$$
D^{0,b}_t= \sum_{0 < s \leq t} \Delta D^{0,b}_s.
$$
On the other hand, it is clear that if $\Delta D^{0,b}_s \neq 0$, then $\Delta X_s >0$. Moreover,  if
$\Delta X_s \geq b $ then $0 \leq\Delta X_s- \Delta D^{0,b}_s \leq b$. 
Thus

\begin{multline*}
    \vert \E_x D^{0,b}_t - \E_x S^+_t  \vert \leq \E_x    \sum_{\ 0 < s \leq t } \left( \Delta X_s \mathbf{1}_{\lbrace 0 < \Delta X_s \leq b \rbrace}  +  b \mathbf{1}_{\{\Delta X_s >b\}}  \right) \\
    = t \int_{(0,b]} y \Pi(dy)+bt \Pi([b,\infty)),
\end{multline*}  
concluding the proof.
\end{proof}
\begin{lemma}\label{L:reflectingboundedlimit}
If the process $ X $ has bounded variation then
\vskip2mm\par\noindent
{\rm(i)}
$\lim_{b \to 0}
\E_x \displaystyle{\int_{(0,\infty)}}e^{-\epsilon s}dD_s^{0,b}=\E_x\displaystyle{\int_{(0,\infty)}}e^{-\epsilon s}  dS^+_s   $, 

\vskip2mm\par\noindent
{\rm(ii)}
$\lim_{b \to b_0} \lim_{T \to \infty} \frac{1}{T}\E_x D^{0,b}_T= \lim_{T \to \infty} \frac{1}{T}\E_x S^+_T $. 
\end{lemma}
\begin{proof}
In the case that the drift is non-positive, both statements follow from Proposition \ref{P:reflectingboundedlimit} (in (i) integration by parts must be used). 
When the drift is positive, similar results can be proven for the process $U_t^{0,b}$,
concluding the proof from the fact $X_t^{0,b} \in [0,b]$ for all $t>0$.
\end{proof}
\subsection{Proofs of Theorems \ref{T:DISCOUNTEDPROBLEMSOLUTION} and \ref{T:ErgodicProblemsolution} }
Finally, we have the necessary ingredients to prove the main results of the article.

\begin{proof}[Proof of Theorem \ref{T:DISCOUNTEDPROBLEMSOLUTION}]
If $c \in C^2(\mathbb{R})$ we consider the function
$$ 
W(x)=\int_{a^\ast_\epsilon}^x V_{\epsilon}(y)dy. 
$$
As a consequence of Proposition \ref{P:DynkinLipschitz}, the function $W \in C^1(\mathbb{R})$.
Furthermore, when  the process has unbounded variation, applying Proposition \ref{L:DynkinC1indense}, we obtain $W \in C^2(\mathbb{R})$.
Finally, when $c\in C^2(\mathbb{R})$, the proof of Theorem \ref{T:DISCOUNTEDPROBLEMSOLUTION}  follows  from Proposition \ref{P:DomainInfinitesimalgenerator} and Lemma \ref{L:HarmonicityPrimitiveDynkin}, taking $u(x)= W(x) +\epsilon^{-1}(\mathcal{L}_XW(a^\ast_\epsilon)+c(a^\ast_\epsilon))$, and applying the verification Theorems  \ref{T:Verificationdiscounted} and \ref{T:VerificationsBarriersOptimalDISCOUNT}. 
Observe that we can assume that the controls $(U,D)$ are under the conditions of  Theorem \ref{T:Verificationdiscounted}, because if the process $X^{U,D}$ does not satisfy inequality \eqref{eq:Conditionforverification}, then $J_\epsilon(x,U,D)=\infty$ due to condition \rm{(ii)} of Definition \ref{D:costfunction} and the fact $G_{\epsilon}$ is linear outside an interval (a similar argument can be found in \cite[Thm. 2.10]{CMO}). 
Furthermore, we deduce that $G_{\epsilon}$ is convex because it is a primitive of the value of the 
$\epsilon$-Dynkin game of Definition  \ref{def:DG}. \\
Assume now that $c \notin C^2 (\mathbb{R})$. 
Note that there exist constants $m_0>0, \ l_0>0 $ 
such that $\vert c'_{\delta}(x)\vert >m_0$ for every $x \notin [-l_0,l_0]$, 
and if $\alpha>0$, 
$\vert c'_{\delta}(l) \vert \wedge \vert c'_{\delta}(-l)\vert \rightarrow \infty$\ $(l\rightarrow  \infty)$ uniformly in $(0,\bar\delta]$ for some $\bar\delta$, because, for $L>1$, we have
\begin{align*}
c'_{\delta}(L)-c'_{\overline{\delta}}(L-1) &\geq \int_{L-1}^L (c'_{\delta} (x)-c'_{\overline{\delta}}(x))dx\\
&=c_{\delta}(L)-c_{\delta}(L-1)-\left(c_{\overline{\delta}}(L)-c_{\overline{\delta}}(L-1)\right)\geq
-4\bar\delta,
\end{align*}
implying that \eqref{E:boundDynkin3} holds uniformly in $(0,\bar\delta]$.
Therefore, Proposition \ref{P:BOUNDSAB} can be used with the same constant $L$ for every $\delta$, thus every pair $(a^{\ast}_{\epsilon, \delta},b^{\ast}_{\epsilon,\delta})$ that defines the Nash Equilibrium for the $\epsilon$-Dynkin game with associated $c_{\delta}'$, belongs to the set  $[-L,L]$. Naturally, taking a subsequence if necessary, we define  $(a^{\ast}_\epsilon,b^{\ast}_\epsilon)$ as the limit of $(a^{\ast}_{\epsilon,\delta},b^{\ast}_{\epsilon,\delta})$ when $\delta \to 0$.  \\
Firstly, assume $x \in (a^{\ast}_\epsilon,b^{\ast}_\epsilon)$ and for  $\delta $ small enough, we can assume $x \in (a^{\ast}_{\epsilon,\delta},b^{\ast}_{\epsilon,\delta})$. On the one hand, using that we previously proved the Theorem when $c \in C^2(\mathbb{R})$, we have for every admissible control $(U,D)$, the inequality:

\begin{align*} %
\lim_{\delta \to 0 }&\E_x \int_{(0,\infty)}e^{-\epsilon s} \big(c_{\delta}(X_s^{a^{\ast}_{\epsilon,\delta},b^{\ast}_{\epsilon,\delta}})ds + q_udU_s^{a^{\ast}_{\epsilon,\delta},b^{\ast}_{\epsilon,\delta}}+ q_ddD_s^{a^{\ast}_{\epsilon,\delta},b^{\ast}_{\epsilon,\delta}} \big)   \nonumber \\
 & \leq \lim_{ \delta \to 0} \E_x\bigg( \int_{(0,\infty)}e^{-\epsilon s} \big(c_{\delta}(X_s^{U,D})ds + q_udU_s+ q_ddD_s \big) +q_u u_0+q_d d_0 \bigg) \nonumber \\
 &=  \E_x\bigg(\int_{(0,\infty)}e^{-\epsilon s} \big(c(X_s^{U,D})ds + q_udU_s+ q_ddD_s \big)+q_u u_0 +q_d d_0  \bigg).
\end{align*}
Again, observe that the inequality holds for every pair of controls $(U,D)$, not only for the ones under the hypothesis of Theorem \ref{T:Verificationdiscounted}.
 Therefore, we need to prove
\begin{multline}\label{E:OptimaldeltaDiscounted}
\lim_{\delta \to 0 }
\E_x \int_{(0,\infty)}e^{-\epsilon s} \big(c(X_s^{a^{\ast}_{\epsilon,\delta},b^{\ast}_{\epsilon,\delta}})ds + q_udU_s^{a^{\ast}_{\epsilon,\delta},b^{\ast}_{\epsilon,\delta}}+ q_ddD_s^{a^{\ast}_{\epsilon,\delta},b^{\ast}_{\epsilon,\delta}} \big)  \\
 =  \E_x \int_{(0,\infty)}e^{-\epsilon s} \big(c(X_s^{a^{\ast}_{\epsilon},b^{\ast}_{\epsilon}}) ds + q_u dU^{a^{\ast}_{\epsilon},b^{\ast}_{\epsilon}}_s+ q_d dD^{a^{\ast}_{\epsilon},b^{\ast}_{\epsilon}}_s \big)  .
\end{multline}
The equality
\[\lim_{\delta \to 0}\E_x \int_{(0,\infty)}e^{-\epsilon s}c(X_s^{a^{\ast}_{\epsilon,\delta},b^{\ast}_{\epsilon,\delta}})ds  =\E_x \int_{(0,\infty)}e^{-\epsilon s} c(X_s^{a^{\ast}_{\epsilon},b^{\ast}_{\epsilon}})ds  \]
is deduced from the fact that the processes $\lbrace X_s^{a^{\ast}_{\epsilon,\delta},b^{\ast}_{\epsilon,\delta}} \rbrace$ and  $\lbrace X_s^{a^{\ast}_{\epsilon},b^{\ast}_{\epsilon}} \rbrace$ are bounded in $[-L,L]$ and \cite[Thm. 2.1]{KRLS}, which implies that for every $A<B, C<D$: 
\begin{equation}\label{Eq:continuitybarrier}
    \vert X^{A,B}_t- X^{C,D}_t \vert \leq \vert X^{A,B}_t- X^{A,D}_t \vert+ \vert X^{A,D}_t- X^{C,D}_t \vert \leq  \vert B-D \vert+ \vert C-A \vert ,
\end{equation}
for all $t\geq 0$. It remains to prove that
\[\lim_{\delta \to 0}\E_x \int_{(0,\infty)}e^{-\epsilon s}   dU^{a^{\ast}_{\epsilon,\delta},b^{\ast}_{\epsilon,\delta}}_s =\E_x \int_{(0,\infty)}e^{-\epsilon s}   dU^{a^{\ast}_{\epsilon},b^{\ast}_{\epsilon}}_s . \]
In the case that the process $X$ has bounded variation, the claim is deduced from Lemma \ref{L:reflectingcontinuity} and Lemma \ref{L:reflectingboundedlimit}. 
In the case of unbounded variation, again we can use Lemma \ref{L:reflectingcontinuity} if $ d= \inf_{\delta}(b_{\epsilon,\delta}^{\ast}-a_{\epsilon,\delta}^{\ast})>0$. 
Assume, by contradiction, that $d=0$. 
Using Lemma \ref{L:reflectingunboundedinfinity} we obtain a subsequence $\lbrace\delta_n\rbrace$ such that 
\begin{equation}\label{Eq:limitU}
\lim_{n \to \infty}
\E_x \displaystyle{\int_{(0,\infty)}}e^{-\epsilon s}  dD_s^{a^{\ast}_{\epsilon,\delta_n},b^{\ast}_{\epsilon,\delta_n}}   =\infty.
\end{equation}
On the other hand, denoting by $G^{\delta}_{\epsilon}$ the $\epsilon$ discounted value function with a cost function $c_{\delta}$ we observe that for every $\delta,\overline{\delta}>0$ the inequality
$\vert \vert G^{\delta}_{\epsilon}-G^{\overline\delta}_{\epsilon}\vert \vert_{\infty} \leq \vert \delta- \overline{\delta} \vert$ holds. 
However, due to the assumption $d=0$ and \eqref{Eq:limitU}  for each $x \in \mathbb{R}$ we have a subsequence $\lbrace  \delta_n \rbrace$ such that 
$\vert G^{\delta_n}_{\epsilon}(x) \vert \to \infty$ what is a contradiction.    \\
Secondly, if $x > b_{\epsilon}^{\ast}$ the inequality \eqref{E:OptimaldeltaDiscounted} becomes 
\begin{align*}
\lim_{\delta \to 0 }&\E_x\bigg(\int_{(0,\infty)}e^{-\epsilon s} \big(c_{\delta}(X_s^{a^{\ast}_{\epsilon,\delta},b^{\ast}_{\epsilon,\delta}})ds + q_udU_s^{a^{\ast}_{\epsilon,\delta},b^{\ast}_{\epsilon,\delta}}+ q_ddD_s^{a^{\ast}_{\epsilon,\delta},b^{\ast}_{\epsilon,\delta}} \big) \\ &+q_d(b^{\ast}_{\epsilon,\delta}-x)   \bigg) \nonumber \\
 & \leq \lim_{ \delta \to 0} \E_x\bigg(\int_{(0,\infty)}e^{-\epsilon s} \big(c_{\delta}(X_s^{U,D})ds + q_udU_s+ q_ddD_s \big) +q_u u_0 +q_d d_0 \bigg) \nonumber \\
 &=  \E_x\bigg(\int_{(0,\infty)}e^{-\epsilon s} \big(c(X_s^{U,D})ds + q_udU_s+ q_ddD_s \big)+q_u u_0+q_d d_0  \bigg),
\end{align*}
and we need to prove 
\begin{multline*}
\lim_{\delta \to 0 }\E_x\bigg(\int_{(0,\infty)}e^{-\epsilon s} \big(c_{\delta}(X_s^{a^{\ast}_{\epsilon,\delta},b^{\ast}_{\epsilon,\delta}})ds + q_udU_s^{a^{\ast}_{\epsilon,\delta},b^{\ast}_{\epsilon,\delta}}+ q_ddD_s^{a^{\ast}_{\epsilon,\delta},b^{\ast}_{\epsilon,\delta}} \big) +q_d(b^{\ast}_{\epsilon,\delta}-x)  \bigg)  \\
 =  \E_x\bigg(\int_{(0,\infty)}e^{-\epsilon s} \big(c(X_s^{a^{\ast}_{\epsilon},b^{\ast}_{\epsilon}})ds + q_udU^{a^{\ast}_{\epsilon},b^{\ast}_{\epsilon}}_s+ q_ddD^{a^{\ast}_{\epsilon},b^{\ast}_{\epsilon}}_s \big) +q_d(b^{\ast}_{\epsilon}-x)\bigg),
\end{multline*}
which is equivalent to prove
\begin{multline}\label{E:OptimaldeltaDiscounted2}
\lim_{\delta \to 0 }\E_{ b^{\ast}_{\epsilon,\delta}}  \int_{(0,\infty)}e^{-\epsilon s} \big(c_{\delta}(X_s^{a^{\ast}_{\epsilon,\delta},b^{\ast}_{\epsilon,\delta}})ds + q_udU_s^{a^{\ast}_{\epsilon,\delta},b^{\ast}_{\epsilon,\delta}}+ q_ddD_s^{a^{\ast}_{\epsilon,\delta},b^{\ast}_{\epsilon,\delta}} \big)    \\
 =  \E_{b^{\ast}_{\epsilon}} \int_{(0,\infty)}e^{-\epsilon s} \big(c(X_s^{a^{\ast}_{\epsilon},b^{\ast}_{\epsilon}})ds + q_udU^{a^{\ast}_{\epsilon},b^{\ast}_{\epsilon}}_s+ q_ddD^{a^{\ast}_{\epsilon},b^{\ast}_{\epsilon}}_s \big) . 
\end{multline}
The proof of $\eqref{E:OptimaldeltaDiscounted2}$ follows the same reasoning as the proof of the previous case with the only precaution that in this case the continuity of $c$ must be used due to the translation of the process. This is because similar analytical properties of the controlled process and reflections in the case where the process starts at the lower barrier can be deduced if the process starts at the upper barrier. The case $x < a^{\ast}_{\epsilon}$ is clearly analogous. \\
Finally, for the case $x \in \lbrace a^{\ast}_{\epsilon},b^{\ast}_{\epsilon} \rbrace$, notice that the function $J_{\epsilon}(x,U^{a^{\ast}_\epsilon,b^{\ast}_\epsilon},D^{a^{\ast}_\epsilon,b^{\ast}_\epsilon}) $
 is equal to  $G_{\epsilon}(x)$  if $x \notin \lbrace a^{\ast}_\epsilon,b^{\ast}_\epsilon \rbrace$. Therefore, from Proposition \ref{P:notdegenerate} we have \newline $ \limsup_{y \to x} J_{\epsilon}(x,U^{a^{\ast}_\epsilon,b^{\ast}_\epsilon},D^{a^{\ast}_\epsilon,b^{\ast}_\epsilon})\leq G_{\epsilon}(x)$ for $x \in \lbrace a^{\ast}_\epsilon,b^{\ast}_\epsilon \rbrace$ and thus, to conclude the proof, we need to show that $$\lim_{x \searrow b^{\ast}_\epsilon}J_{\epsilon}(x,U^{a^{\ast}_\epsilon,b^{\ast}_\epsilon},D^{a^{\ast}_\epsilon,b^{\ast}_\epsilon})= J_{\epsilon}(b^{\ast}_\epsilon,U^{a^{\ast}_\epsilon,b^{\ast}_\epsilon},D^{a^{\ast}_\epsilon,b^{\ast}_\epsilon}) .$$
This claim is obtained from the fact that, if $x \geq b^{\ast}_{\epsilon},$, then  
 $$J_{\epsilon}(x,U^{a^{\ast}_\epsilon,b^{\ast}_\epsilon},D^{a^{\ast}_\epsilon,b^{\ast}_\epsilon})= q_d(x-b^{\ast}_\epsilon)+ J_{\epsilon}(b^{\ast}_\epsilon,U^{a^{\ast}_\epsilon,b^{\ast}_\epsilon},D^{a^{\ast}_\epsilon,b^{\ast}_\epsilon}) .$$

\end{proof}
\begin{proof}[Proof of Theorem \ref{T:ErgodicProblemsolution}.] 
First, observe that to prove \rm{(i)} it is enough to show the pointwise convergence in \rm{(i)} and the claims \rm{(ii)} and \rm{(iii)}, due to the convexity of $G_{\epsilon}$ (see for example \cite[Thm. 10.8]{Rockafellar}). Furthermore let us show that it is sufficient to prove \rm{(i)} for twice continuously differentiable functions: for every $\delta >0,\epsilon>0$ denote $G_{\epsilon}^{\delta}(x),G^{\delta}(x)$ the $\epsilon$-discounted value function and ergodic value function respectively with underlying cost $c_{\delta}$. Observe
\[ \vert \epsilon G_{\epsilon}^{\delta}(x)- \epsilon G_{\epsilon}(x) \vert  \leq  \delta, \ \vert G^{\delta}(x)-G(x) \vert \leq \delta, \ \text{for all } x \in \mathbb{R}. \]
Therefore, if \[\limsup_{\epsilon \to 0} \vert \epsilon G_{\epsilon}^{\delta}(x)-G^{\delta}(x) \vert =0, \ \text{for all } \delta>0 ,\]
item $\rm{(i)}$ will hold. Thus, to prove $\rm{(i)}$, we can assume $c \in C^{2}(\mathbb{R})$. \\ 
Take $r >0$ and $(U,D) \in \mathcal{A}$ such that
\[ J(x,U,D) \leq G(x)+r .\]
Using (ii) in Definition \ref{D:costfunction}, we deduce that there is a constant $K>0$ such that
\begin{equation}\label{E:ConvergenceErgodicleftlimit0}
\limsup_{T \to \infty} \frac{1}{T} \E_x  \int_0^T ( \vert X_s^{U,D} \vert +\vert x \vert )  ds <K.
\end{equation}
On the other hand, observe that for each $\epsilon>0$ the function $G_{\epsilon}$ 
satisfies the hypothesis of Theorem \ref{T:VerificationsBarriersOptimal}, thus:
\begin{equation*}
\liminf_{T \to \infty} \frac{\epsilon}{T} \E_x \int_0^T G_{\epsilon}(X^{U,D}_s)ds  \leq G(x)+r   \leq \limsup_{T \to \infty} \frac{\epsilon}{T} \E_x \int_0^T G_{\epsilon}(X^{a^*_{\epsilon},b^*_{\epsilon}}_s) ds  +r, 
\end{equation*} 
 with $(a^*_{\epsilon},b^*_{\epsilon})$ the optimal barriers of the $\epsilon$ Dynkin game defined in Section $2$.  We deduce, using the fact $r$ is arbitrary, it is enough to prove:
\begin{equation}\label{E:ConvergenceErgodicleftlimit}
\liminf_{\epsilon \to 0} \left( \liminf_{T \to \infty} \frac{\epsilon}{T} \E_x \int_0^T G_{\epsilon}(X^{U,D}_s)ds  -\epsilon G_{\epsilon}(x) \right)  \geq 0 ,
 \end{equation}
\begin{equation}\label{E:ConvergenceErgodicrightlimit}
\limsup_{\epsilon \to 0} \left( \limsup_{T \to \infty} \frac{\epsilon}{T} \E_x \int_0^T G_{\epsilon}(X^{a_{\epsilon},b_{\epsilon}}_s)ds  -\epsilon G_{\epsilon}(x) \right) =0 . 
 \end{equation} 
The limit \eqref{E:ConvergenceErgodicrightlimit} is obtained because $X_s^{a,b}$ is bounded and the derivative of $G_{\epsilon}$ is bounded in $[-q_u,q_d]$. 
To prove \eqref{E:ConvergenceErgodicleftlimit}, we use \eqref{E:ConvergenceErgodicleftlimit0}:
\begin{multline*}
\liminf_{\epsilon \to 0}  \liminf_{T \to \infty} \E_x \int_0^T \left( \frac{\displaystyle \epsilon G_{\epsilon}(X^{U,D}_s) -\epsilon G_{\epsilon}(x)}{T} \right)  ds    \\
\geq \liminf_{\epsilon \to 0} \liminf_{T \to \infty} \E_x  \epsilon \int_0^T \left( \frac{\displaystyle  (-q_d-q_u)\left( \vert (X^{U,D}_s)\vert + \vert x\vert \right)}{T} \right) ds      \\
\geq \liminf_{\epsilon \to 0} \liminf_{T \to \infty}   \epsilon \left(-q_d-q_u \right)K =0 .
\end{multline*}
We conclude that the pointwise convergence holds because $r$ is arbitrary. 

The claim \rm{(ii)} follows from the fact that $ \epsilon \vert G_{\epsilon}^{\delta}(x)-G_{\epsilon}^{\delta}(y) \vert \leq (q_d+q_u) \epsilon \vert x -y \vert $.

To prove \rm{(iii)} first assume $c \in C^2( \mathbb{R})$. 
We define $a^*$ and $b^*$ as the limit when $\epsilon \to 0$ of, respectively,
$\lbrace a^*_{\epsilon}\rbrace_{\epsilon>0}$ and $\lbrace b^*_{\epsilon}\rbrace_{\epsilon>0}$,
taking a sequence if necessary.
Using that the pointwise convergence in \rm{(i)} holds and Theorem \ref{T:VerificationsBarriersOptimal}, we deduce
\begin{equation*}
 \limsup_{\epsilon \to 0}  J(x,U^{a^{\ast}_{\epsilon},b^{\ast}_{\epsilon}},D^{a^{\ast}_{\epsilon},b^{\ast}_{\epsilon}}) = G(x). \end{equation*}
Thus it is enough to prove
\begin{equation*}
 \lim_{\epsilon \to 0} J(x,U^{a^{\ast}_{\epsilon},b^{\ast}_{\epsilon}},D^{a^{\ast}_{\epsilon},b^{\ast}_{\epsilon} }) =  J(x,U^{a^{\ast},b^{\ast}},D^{a^{\ast},b^{\ast}}) . \end{equation*}
 Again (taking a sequence if necessary), 
 we can assume $ d_{\epsilon}\colon = (b^{\ast}_{\epsilon}- a^{\ast}_{\epsilon}) $ is monotonic. From \eqref{Eq:continuitybarrier}, and Lemmas   \ref{L:reflectingcontinuity}, \ref{L:reflectingunboundedinfinity} and \ref{L:reflectingboundedlimit}, we conclude the claim in a similar way as in the previous Theorem. \\
 To conclude the proof when the cost function is not twice continuously differentiable, 
 we use that the optimal barriers for $(a^{\ast}_{\delta},b^{\ast}_{\delta})$ belong to a compact set and use \eqref{Eq:continuitybarrier}, and the results of the previous subsection.
\end{proof}

\section{Examples}\label{S4:Examples} 
In this section, we apply our main results.
According to Theorems  \ref{T:DISCOUNTEDPROBLEMSOLUTION} and \ref{T:ErgodicProblemsolution},
the solution is found within the class of reflecting controls. 
To compute the corresponding value functions, %
for stable processes we elaborate on the escape probability formula  \cite[eq. (6)]{AAGP}
and the stationary distribution obtained in  \cite[pg. 155, line 3]{AAGP};
and  for the Compound Poisson processes (with or without Gaussian part), we rely on the memoryless property of the exponential distribution to work with the overshoot and undershoot (see \cite[Theorem 3.1]{CCW}).
\subsection{Introduction}
In  \cite[Prop. 5.1]{AAGP}, the authors prove for a L\'evy process as in \eqref{D:LEVYPROCESSEQUATION}, that if $\P_s$ is the distribution of $X_s^{a,b}$ then 
$\| \P_s (x, \cdot ) - \pi^{a,b}(\cdot) \|$ 
converges to zero in the norm of the total variation, where
\begin{equation}\label{E:stationary}
\pi^{a,b}[x,b]=\P(X_{\eta_{[x-b,x-a)^c}}\geq x-a ) 
\end{equation}
is the stationary measure of the process $X^{a,b}$ and $\eta_{[x-b,x-a)^c}$ denotes the first entry to the set $ [x-b,x-a)^c$.
Furthermore, in Theorem 1.1 of the same article, 
when $\mu=\E X_1<\infty$,
the following relationship is established:
\begin{equation*}
\E_\pi D^{a,b}_1
 =\frac{1}{b-a} \left( 2 \mu\E_{\pi} X^{a,b}_1+\nu^2+ \int_0^{b-a}\pi^{0,b-a}(dx)\int_{-\infty}^{\infty}\varphi(x,y,b) \Pi(dy)\right),  
\end{equation*}
with
\[ 
\varphi(x,y,b)=
\begin{cases}
-(x^2+2xy), & \mbox{if $y \leq -x$},\\
        y^2, & \mbox{if $-x <y<b-x $},\\
        2y(b-x)-(b-x)^2, & \mbox{if $y \geq b-x$}.
\end{cases} 
\] 
We know  by stationarity, that 
$$
\lim_{T\to\infty}\frac1T\E_x U^{a,b}_T =\E_\pi U^{a,b}_1,\quad \lim_{T\to\infty}\frac1T\E_x D^{a,b}_T=\E_\pi D^{a,b}_1,
$$
Consequently, from \eqref{D:controlledequation}, we obtain
$$
\E_\pi D_1^{a,b}=\mu+\E_\pi U_1^{a,b}
$$
giving the following result.
\begin{lemma}\label{L:ergodiccost} 
Under the assumptions given in the introduction, for $a<x<b$, and $d=b-a$, the
ergodic cost function of Definition \ref{D:ergodicvaluefunction} satisfies
\begin{align}
J(x,(U^{a,b},D^{a,b}))
&=   \int_{[a,b]} c(u)  \pi^{a,b}(du)  +\E_\pi(q_dD^{a,b}_1+q_uU^{a,b}_1)\nonumber \\
&=  \int_{[0,d]}c(u+a) \pi^{0,d}(du) +q\E_\pi D^{0,d}_1 -\mu q_u.\label{E:secondline}
\end{align}
\end{lemma} 
In \eqref{E:secondline}, the lower point $a$ only appears in the integral in the cost, 
and the second variable $d=b-a$ is the distance within the barriers.  Condition $a^*\leq 0\leq b^*$ now reads
$0\leq -a^*\leq d^*$.
In what respects to the discounted problem, 
it is reduced to finding a couple $a^*\leq 0\leq b^*$
that minimizes
\begin{equation*}
J_\epsilon(x,a,b)=\E \int_{(0,\infty)}e^{-\epsilon s}\left(c(X_s^{a,b})ds + q_udU^{a,b}_s +q_d dD^{a,b}_s\right).  
\end{equation*}
Notice that the process starts at zero because as seen in Section \ref{S2:DynkinAdjoint} the optimal reflecting barriers do not depend on the starting point.
\subsection{Ergodic problem for a Compound Poisson Process}\label{ExamplePoisson}
In this example, the cost function is $c(x)= \vert x \vert$.  
In this case $c_\delta$ can be taken as
\[
c_{\delta}(x)=2 \delta \log(1+ e^{\delta^{-1}x})-x -2\delta \log 2.
\]
We consider a compound Poisson process $X=\{X_t\}_{t\geq 0}$ with double-sided exponential jumps, given by
\begin{equation*}
X_t=x+\sum_{i=1}^{N^{(1)}_t}Y^{(1)}_i-\sum_{i=1}^{N^{(2)}_t}Y^{(2)}_i,
\end{equation*}
where $\{N^{(1)}_t\}_{t\geq 0}$ and 
$\{N^{(2)}_t\}_{t\geq 0}$ 
are two Poisson processes with respective positive intensities $\lambda_1,\lambda_2$;
$\{Y^{(1)}_i\}_{i\geq 1}$ and $\{Y^{(2)}_i\}_{i\geq 1}$ 
are two sequences of independent exponentially distributed random variables with respective positive parameters 
$\alpha_1,\alpha_2$.
The four processes are independent.
Consequently
$$
\phi(z)=\lambda_1{z\over\alpha_1-z}-\lambda_2{z\over\alpha_2+z}.
$$
For definiteness, we assume 
$
\E X_1={\lambda_1/\alpha_1}-{\lambda_2/\alpha_2} <0.
$
Consider the Lundberg  constant ${{\rho}}$, i.e the positive root of $\phi(z)=0$, given by
$$
{{\rho}}={\lambda_2\alpha_1-\lambda_1\alpha_2\over\lambda_1+\lambda_2}.
$$
Observe that $0<{{\rho}}<\alpha_1$. To compute \eqref{E:secondline} we obtain the stationary distribution 
based on \eqref{E:stationary} and the fact that $\{\exp({\rho} X_t)\}_{t\geq 0}$ is a martingale:
$$
\aligned
\pi^{0,d}(dx)&=
{{\rho}/\lambda_1\over\alpha_1/\lambda_1-\alpha_2e^{{-\rho}d}/\lambda_2}
\delta_0(dx)
+
{(\alpha_1+\alpha_2)/(\lambda_1+\lambda_2)\over\alpha_1/\lambda_1-\alpha_2e^{{-\rho}d}/\lambda_2}\rho e^{-\rho x}dx\\
&\quad+
{{\rho}/\lambda_2\over\alpha_1e^{{\rho}d}/\lambda_1-\alpha_2/\lambda_2 }
\delta_d(dx),\\
\endaligned
$$
where $\delta_a(dx)$ is the Dirac measure at $a$.
We minimize \eqref{E:secondline} from Lemma \ref{L:ergodiccost}. 
After computing $\int_{[0,d]}\vert a+u \vert \pi^{0,d}(du)$, we deduce that the cost function to be minimized, 
after the change $d=b-a$, is 
\begin{multline*}
J(a,d)= 
   \Bigg(  \frac{-a\rho}{\lambda_1}+e^{-\rho d}(d+a)\frac{\rho}{\lambda_2}   
+\frac{q\rho\lambda_1}{\alpha_1}  \left(e^{-\rho d} \left(\lambda_1^{-1} + \lambda_2^{-1} \right)   \right) \\
\qquad +\left( \frac{\alpha_1+\alpha_2}{\lambda_1+\lambda_2}\right)\left(e^{-\rho d}\left(- a - d  -\rho^{-1} \right) -a  +2\rho^{-1} e^{a \rho}-\rho^{-1} \right)\Bigg) \\
\times \left(\frac{\alpha_1}{\lambda_1}-\frac{\alpha_2e^{-{\rho}d}}{\lambda_2}\right)^{-1}.
\end{multline*}
Upon inspection, we conclude that the candidates for the optimal barriers are  
$a=0, d=0 $ or $d$ minimizing the expression above with $a$ equal to $0 , -d$ or the value $$a=-\rho^{-1}\log\left( 2(\alpha_1+\alpha_2) \right)+ \rho^{-1} \log \left(e^{-\rho d}(\alpha_2 +\lambda_1\alpha_2/\lambda_2)+\alpha_1+\lambda_2\alpha_1/\lambda_1 \right) .$$
To illustrate, we consider numeric examples showing that the four possible cases can happen:
\begin{align*}
    &\text{If } \lambda_1=1, \ \lambda_2=2,\alpha_1=2,\ \alpha_2=1,\ q=3, \text{ then } a^{\ast}=0,\  d^{\ast}\sim4.005 .  \\
    &\text{If } \lambda_1=1, \ \lambda_2=2,\ \alpha_1=2,\alpha_2=1, \ q=0.1, \text{ then } a^{\ast}=0, \ d^{\ast}=0. \\
    &\text{If } \lambda_1=1, \ \lambda_2=1,\ \alpha_1=4,\alpha_2=1, \ q=5, \text{ then } a^{\ast} \sim -0.272, \ d^{\ast} \sim 0.966.
  \end{align*}
The last case, giving $a^{\ast}=-d^{\ast}\neq 0$ is very sensitive to the parameter variations
  \begin{align*}
     &\text{If } \lambda_1=9.999985 \times 10^5, \ \lambda_2=1, \ \alpha_1=1 \times 10^6 ,\ \alpha_2=5\times 10^{-7},\\ & 
     q=0.5 , \text{ then } -a^{\ast}=d^{\ast}\sim 0.0202 .
   \end{align*}
In a certain sense, the case where the optimum is reached at zero and a negative barrier can be seen as pathological, because the expected value is negative, yet it is optimal to keep the process in the negative line. \\ 
\subsection{Discounted problem for a jump-diffusion process}\label{ExamplePoissonGaussian}
In this case, $\epsilon>0$ is fixed, $c(x)=x^2/2$. 
We assume that
the L\'evy process $\{X_t\}_{t\geq 0}$ has non-zero mean defined by
\begin{equation*}
X_t=x+\nu W_t+\sum_{i=1}^{N_t^{(1)}}Y^{(1)}_i  - \sum_{i=1}^{N_t^{(2)}}Y^{(2)}_i,
\end{equation*} 
with $\{W_t\}_{t\geq 0} $ a Brownian motion, $\nu>0$, and
$\{N^{(1)}_t\}_{t\geq 0}$,$\{N^{(2)}_t\}_{t\geq 0}$,$\lbrace Y^{(1)}_i \rbrace_{i\geq 1}$,
$\lbrace Y^{(2)}_i \rbrace_{i\geq 1}$ 
as in the previous example, the five processes are independent (for more information about the first exit time of an interval for this process, see \cite{CCW}). Therefore,

$$
\phi(z)=\frac{\nu^2 }{2}z^2+ \lambda_1{z\over\alpha_1-z}-\lambda_2{z\over\alpha_2+z}.
$$
Now we find a pair $a^{\ast} \leq b^{\ast}$ such that 
\[
G_{\epsilon}(x)=J_{\epsilon}(x,(U^{a^{\ast},b^{\ast}},D^{a^{\ast},b^{\ast}})). 
\]
Taking $\tau (a), \ \sigma(b)$ as in \eqref{D:tauandsigma}

\begin{multline*}
M_x(a,b)=\E_x\bigg(
\int_0^{\tau(a)\wedge\sigma(b)}
e^{-\epsilon s}X_s ds \\
-q_u e^{-\epsilon\tau(a)}\mathbf{1}_{\{\tau(a)<\sigma(b)\}}  
+q_d e^{-\epsilon\sigma(b)}\mathbf{1}_{\{\sigma(b)<\tau(a)\}} 
\bigg),
\end{multline*}
and applying Theorem \ref{T:DISCOUNTEDPROBLEMSOLUTION}, 
to solve the discounted control problem,
we need to find 
  $a^{\ast}<0<b^{\ast}$ such that
\begin{equation*}
M(a^*,b^*)=\sup_{a<0}\,\inf_{b>0}M(a,b),
\end{equation*}
where we assume $X_0=0$ because $(a^{\ast},b^{\ast})$ do not depend on the starting point. 
Using \cite[Thm. 3.1]{CCW}, we deduce that $M(a,b)$ is equal to:

\begin{align*}
&(e^{-\rho_3 b}, e^{-\rho_4 b},e^{-\rho_1 a}, e^{-\rho_2 a})\textbf{N}^{-1}_{b-a}\begin{pmatrix}
-\epsilon^{-2}(\lambda_1/\alpha_1 -\lambda_2/\alpha_2)-\epsilon^{-1}b +q_d \\
-\alpha_1^{-1}(\epsilon^{-2}(\lambda_1/\alpha_1 -\lambda_2/\alpha_2)+\epsilon^{-1} \alpha_1^{-1}+\epsilon^{-1}b -q_d)\\
-\epsilon^{-2}(\lambda_1/\alpha_1 -\lambda_2/\alpha_2)-\epsilon^{-1}a -q_u \\
-\alpha_2^{-1}(\epsilon^{-2}(\lambda_1/\alpha_1 -\lambda_2/\alpha_2)-\epsilon^{-1} \alpha_2^{-1}+\epsilon^{-1}a +q_u)
\end{pmatrix} \\&+ \frac{\displaystyle \lambda_1/ \alpha_1 - \lambda_2/\alpha_2}{\epsilon^2},
\end{align*}
with
\begin{equation*}
\textbf{N}_{b-a} = 
 \begin{pmatrix}
1 & 1 & e^{- \rho_1 (a-b)} & e^{-\rho_2 (a-b)} \\
\frac{\displaystyle 1 }{\displaystyle \alpha_1 -\rho_3} & \frac{\displaystyle 1}{\displaystyle \alpha_1 -\rho_4} & \frac{\displaystyle e^{-\rho_1 (a-b)}}{\displaystyle \alpha_1 -\rho_1} & \frac{\displaystyle e^{-\rho_2 (a-b)}}{\displaystyle \alpha_1- \rho_2} \\
e^{\rho_3 (a-b)} & e^{\rho_4 (a-b)} & 1 & 1 \\
\frac{\displaystyle e^{\rho_3 (a-b)}}{\displaystyle \alpha_2+\rho_3} & \frac{\displaystyle e^{\rho_4 (a-b)} }{\displaystyle \alpha_2+\rho_4} & \frac{\displaystyle 1}{\displaystyle \alpha_2+ \rho_1}   & \frac{\displaystyle 1}{\displaystyle \alpha_2 + \rho_2}
\end{pmatrix},
\end{equation*}
where ${\rho}_i\ (i=1,2,3,4)$ are the non-zero roots of the equation $\phi(z)=\epsilon$, that satisfy  $\rho_2<-\alpha_2<\rho_1<0<\rho_3< \alpha_1 < \rho_4. $ This matrix is always non-singular (see \cite[Prop. 3.1]{CCW}). \\
For the parameters:
$$q_d=1, \
q_u=1, \
\alpha_1=2, \
\alpha_2=1, \
\lambda_2=1, \
\lambda_1= 1, \
\epsilon=1, \
\nu=\sqrt{2}, $$
the solutions of the equation $\phi(z)-\epsilon=0$ are:

$$\rho_1\sim -0.489, \  \rho_2 \sim -1.898, \ \rho_3\sim 0.849,  \ \rho_4 \sim 2.537 ,$$
and the equilibrium point is  
$ a^{\ast}\sim -2.017      , b^{\ast} \sim  2.311   $.
\subsection{Ergodic problem for a stable processes}
In this example $c(x)=x^2$, and the L\'evy process $X$ is strictly $\alpha-$stable  with parameter $\alpha \in (1,2), 0<c^+<c^-$. 
In other words $X$ is a pure jump process with finite mean and 
\[ 
\Pi(dx)= 
\begin{cases}
c^{+} x^{-\alpha -1}dx, & x>0, \\
c^{-} \vert x \vert^{-\alpha -1} dx ,& x<0.
\end{cases}
\]
The characteristic exponent  is
\[
\phi(i \theta) \colon =\vert \theta \vert^{\alpha}(c^+ + c^-) \Big(1- i \text{sgn}(\theta)\tan \big(\frac{\pi \alpha}{2} \big) \frac{c^+ -c^-}{c^+ + c^-} \Big) .
\]
(see \cite[p. 10]{KRI}). 
The two sided exit problem can be solved using the scaling property, that is 
$X_t\overset{d}{=} k^{\frac{-1}{\alpha}}X_{tk}$ for every $t>0, k>0$. 
The stationary measure has Beta density $\pi^{0,d}(x)$ on $[0,d]$ with parameters
$(\alpha \rho, \alpha (1-\rho))$, i.e.
\begin{equation*}
\pi^{0,d}(x)=\frac1{d \beta(\alpha \rho, \alpha (1-\rho))} \left(1-\frac{x}{d} \right)^{\alpha \rho -1} 
\left(\frac{x}{d} \right)^{\alpha (1-\rho)-1}, 
\end{equation*}
where $\beta(u,v)$
is the Beta function and 
\[
\rho= \frac{1}{2}+ (\pi \alpha)^{-1} \arctan \bigg( \Big(\frac{c^+ - c^-}{c^+ + c^-} \Big)  \tan(\alpha \pi/2) \bigg). 
\]
From this, we can compute $\E D^{0,d}_1$ (see \cite[p. 114]{AAGP}) and deduce, by differentiation $$  a=- \int_0^d x \pi^{0,d}(x)dx =-d\rho.$$
Then, the ergodic problem has an objective function depending only on $d$, given by
\begin{equation*}
J(d)  =\int_0^d (x-d\rho)^2 \pi^{0,d}(x)\,dx +q \E D^{0,d}_1
        =d^2{\rho(1-\rho)\over\alpha+1}+\frac{1}{d^{\alpha-1}}q\E D^{0,1}_1.    
\end{equation*}
By differentiation:
$$
d^*=\left(\frac{(\alpha^2-1)q \E D^{0,1}_1}{2\rho(1-\rho)}\right)^{1/(\alpha+1)}.
$$
For example, when $q=1, \ c^-=2, \ c^+=1, \alpha=1.5$ the ergodic value is reached at 
$d^{\ast}  \sim 2.850$ and 
$a^{\ast}=-1.230$.

\end{document}